\documentclass[11pt,letterpaper]{amsart}

\usepackage[T1]{fontenc}

\usepackage{amssymb}
\usepackage{amsmath}
\usepackage{amsfonts}
\usepackage{bbm}
\usepackage{amsthm}
\usepackage{lmodern}
\usepackage{mathrsfs}
\usepackage[colorlinks=true, linkcolor=cyan!20!blue, citecolor=purple, urlcolor=black]{hyperref}
\usepackage[capitalise]{cleveref}

\usepackage{color}
\usepackage[margin=3.4cm]{geometry}

\usepackage[all,cmtip]{xy}
\usepackage[utf8]{inputenc}

\usepackage{graphicx}

\usepackage{varwidth}
\usepackage{overpic}

\makeatletter
\DeclareRobustCommand{\em}{%
	\@nomath\em \if b\expandafter\@car\f@series\@nil
	\normalfont \else \slshape \fi}
\makeatother
\usepackage{upgreek}	
\usepackage{rotating}
\usepackage{tikz}
\usetikzlibrary{matrix,arrows,decorations.pathmorphing, shapes.geometric}
\usepackage{needspace}

\usetikzlibrary{decorations.markings}

\usepackage{todonotes}

\usepackage{tikz-cd}

\usetikzlibrary{decorations.markings}
\usetikzlibrary{backgrounds}

\tikzstyle{tikzfig}=[baseline=-0.25em,scale=0.5]

\pgfkeys{/tikz/tikzit fill/.initial=0}
\pgfkeys{/tikz/tikzit draw/.initial=0}
\pgfkeys{/tikz/tikzit shape/.initial=0}
\pgfkeys{/tikz/tikzit category/.initial=0}

\pgfdeclarelayer{edgelayer}
\pgfdeclarelayer{nodelayer}
\pgfsetlayers{background,edgelayer,nodelayer,main}

\tikzstyle{none}=[inner sep=0mm]

\newcommand{\tikzfig}[1]{%
	{\tikzstyle{every picture}=[tikzfig]
		\IfFileExists{#1.tikz}
		{\input{#1.tikz}}
		{%
			\IfFileExists{./figures/#1.tikz}
			{\input{./figures/#1.tikz}}
			{\tikz[baseline=-0.5em]{\node[draw=red,font=\color{red},fill=red!10!white] {\textit{#1}};}}%
	}}%
}

\tikzstyle{every loop}=[]

\usepackage{tikzit}
\tikzstyle{black dot}=[fill=black, draw=black, shape=circle, minimum size=3pt, inner sep=0pt]
\tikzstyle{black dot small}=[fill=black, draw=black, shape=circle, minimum size=2pt, inner sep=0pt]
\tikzstyle{fblack dot}=[fill=black, draw=red, shape=circle, minimum size=2pt, inner sep=0pt]
\tikzstyle{wbox}=[fill=white, draw=black, shape=rectangle, minimum height=0.5cm, minimum width=0.01cm]
\tikzstyle{bbox}=[fill=white, draw=blue, shape=rectangle, minimum height=0.5cm, minimum width=0.01cm]
\tikzstyle{rbox}=[fill=white, draw=red, shape=rectangle, minimum height=0.5cm, minimum width=0.01cm]
\tikzstyle{bwbox}=[draw=blue, shape=rectangle, minimum width=2cm, minimum height=0.5cm]
\tikzstyle{bbwbox}=[draw=blue, shape=rectangle, minimum width=1cm, minimum height=1cm]
\tikzstyle{big white circle}=[fill=white, draw=black, shape=circle, minimum width=0.75cm]
\tikzstyle{white dot big}=[fill=white, draw=black, shape=circle, inner sep=1pt]
\tikzstyle{white dot}=[fill=white, draw=black, shape=circle, minimum size=3pt, inner sep=0pt]
\tikzstyle{flat box}=[fill=white, draw=black, shape=rectangle, minimum width=1.3cm, minimum height=0.5cm,fill=morphismcolor]
\tikzstyle{square}=[fill=white, draw=black, shape=rectangle]
\tikzstyle{flat box 2}=[fill=white, draw=black, shape=rectangle, minimum height=0.5cm, minimum width=0.01cm,fill=morphismcolor]
\tikzstyle{bigbox}=[fill=white, draw=black, shape=rectangle, minimum height=0.5cm, minimum width=0.8cm,fill=white]
\tikzstyle{over }=[front]
\tikzstyle{theta}=[fill=blue, draw=blue, shape=ellipse, minimum height=6pt, minimum width=6pt, inner sep=0pt]
\tikzstyle{thetabig}=[fill=blue, draw=blue, shape=ellipse, minimum width=1cm, minimum height=0.01cm]
\tikzstyle{thetainv}=[fill=blue, draw=red, shape=ellipse, minimum height=6pt, minimum width=6pt, inner sep=0pt]
\tikzstyle{thetabinv}=[fill=blue, draw=red, shape=ellipse, minimum width=1cm, minimum height=0.01cm]
\tikzstyle{bigdisk}=[draw=black, shape=circle, minimum width=3cm]
\tikzstyle{wdisk}=[shape=circle, minimum width=0.48cm,fill=white]
\tikzstyle{bigdisk2}=[draw=black, fill=lightgray, shape=circle, minimum width=3cm]
\tikzstyle{little disk}=[fill=white, draw=black, shape=circle, minimum width=0.5cm]

\tikzstyle{mid arrow}=[-, postaction={on each segment={mid arrow}}]
\tikzstyle{end arrow}=[->]
\tikzstyle{mover}=[-, link]
\tikzstyle{string}=[-, draw=blue,postaction={on each segment={mid arrow}}]
\tikzstyle{stringd}=[-, dotted,draw=blue,postaction={on each segment={mid arrow}}]
\tikzstyle{red}=[-,draw=red]
\tikzstyle{mydots}=[-,dotted,dashed,draw=gray]
\tikzstyle{mydotsblack}=[-,dotted,dashed,draw=black]
\tikzstyle{open}=[-, line width=1pt,draw=blue]
\tikzstyle{thick}=[-,line width=1pt]
\tikzstyle{rarrow}=[->,draw=red]
\tikzstyle{red mid arrow}=[-, draw={rgb,255: red,214; green,42; blue,51}, postaction={on each segment={mid arrow}}, line width=1pt]
\tikzstyle{RED}=[-, draw={rgb,255: red,214; green,42; blue,51}]
\tikzstyle{REDdashed}=[-,dashed, draw={rgb,255: red,214; green,42; blue,51}]
\tikzstyle{REDarrow}=[->, draw={rgb,255: red,214; green,42; blue,51}]
\tikzstyle{darrow}=[->,dotted]
\tikzstyle{blue}=[-, draw=blue]
\tikzstyle{blue mid arrow}=[-, draw={rgb,255: red,23; green,37; blue,167}, postaction={on each segment={mid arrow}}, line width=1pt]
\tikzstyle{over}=[-, link]
\tikzstyle{mover}=[-, link]
\tikzstyle{mapsto}=[{|->}]

\usetikzlibrary{decorations.pathreplacing,decorations.markings}
\tikzset{
	on each segment/.style={
		decorate,
		decoration={
			show path construction,
			moveto code={},
			lineto code={
				\path [#1]
				(\tikzinputsegmentfirst) -- (\tikzinputsegmentlast);
			},
			curveto code={
				\path [#1] (\tikzinputsegmentfirst)
				.. controls
				(\tikzinputsegmentsupporta) and (\tikzinputsegmentsupportb)
				..
				(\tikzinputsegmentlast);
			},
			closepath code={
				\path [#1]
				(\tikzinputsegmentfirst) -- (\tikzinputsegmentlast);
			},
		},
	},
	mid arrow/.style={postaction={decorate,decoration={
				markings,
				mark=at position .7 with {\arrow[#1]{stealth}}
	}}},
}
\tikzset{%
	link/.style    = { white, double = black, line width = 1.8pt,
		double distance = 0.4pt },
	channel/.style = { white, double = black, line width = 0.8pt,
		double distance = 0.8pt },
}

\tikzstyle{tikzfig}=[baseline=-0.25em,scale=0.5]

\pgfkeys{/tikz/tikzit fill/.initial=0}
\pgfkeys{/tikz/tikzit draw/.initial=0}
\pgfkeys{/tikz/tikzit shape/.initial=0}
\pgfkeys{/tikz/tikzit category/.initial=0}

\pgfdeclarelayer{edgelayer}
\pgfdeclarelayer{nodelayer}
\pgfsetlayers{background,edgelayer,nodelayer,main}

\tikzstyle{none}=[inner sep=0mm]

\tikzstyle{every loop}=[]

\usepackage{mathtools}
\usepackage{epstopdf}

\newtheorem{theorem}{Theorem}[subsection]
\newtheorem{corollary}[theorem]{Corollary}
\newtheorem{lemma}[theorem]{Lemma}
\newtheorem{proposition}[theorem]{Proposition}
\newtheorem{conj}[theorem]{Expectation}
\newtheorem{question}[theorem]{Question}

\newtheorem{innercustomthm}{Theorem}
\newenvironment{customthm}[1]
{\renewcommand\theinnercustomthm{#1}\innercustomthm}
  {\endinnercustomthm}
  
\theoremstyle{definition}
\newtheorem{definition}[theorem]{Definition}
\newtheorem{remark}[theorem]{Remark}
\newtheorem{exx}[theorem]{Example}

\usepackage{enumitem}

\definecolor{Blue}  {rgb} {0.282352,0.239215,0.803921}
\definecolor{Green} {rgb} {0.133333,0.545098,0.133333}
\definecolor{Red}   {rgb} {0.803921,0.000000,0.000000}
\definecolor{Violet}{rgb} {0.580392,0.000000,0.827450}

\newcounter{jfc}
\newcounter{jfctodoab}
\newcounter{jfctodolw}

\usepackage{dingbat}

\newcommand{\cat}[1]{\mathcal{#1}}
\newcommand{\catf}[1]{\mathsf{#1}}

\newcommand{\Graphs}{\catf{Graphs}}
\newcommand{\Legs}{\catf{Legs}}
\newcommand{\Surf}{\catf{Surf}}
\newcommand{\Teich}{\catf{Teich}}

\newcommand{\Rexf}{\catf{Rex}^\catf{f}}

\newcommand{\vect}{\catf{vect}}

\newcommand{\Map}{\catf{Map}}

\def\Cat{\catf{Cat}}
\def\End{\catf{End}}
\newcommand{\Ac}{\cat{A}}
\newcommand{\Tc}{\cat{T}}
\newcommand{\Cc}{\cat{C}}

\newcommand{\Dd}{\mathbb{D}}

\newcommand{\Gpd}{\catf{Grpd}}
\newcommand{\EndAk}{\End^\Cc_\kappa}

\newcommand{\ra}[1]{\xrightarrow{\   #1    \ }}

\def\Hom{\mathrm{Hom}}
\def\op{\mathrm{op}}
\def\id{\mathrm{id}}

\usepackage{scalerel}
\usepackage{stackengine}
\stackMath

\newcommand{\simeqaccent}[1]{%
  \stackengine{0ex}{#1}{%
    \smash{\ensurestackMath{%
      \stackengine{0.2ex}{\overline{\phantom{#1}}}{\widetilde{\phantom{#1}}}{O}{c}{F}{F}{L}%
    }}%
  }{O}{c}{F}{F}{L}%
}

\newcommand{\triangleaccent}[1]{%
  \stackengine{0ex}{#1}{%
    \smash{\ensurestackMath{%
      \stackengine{0.2ex}{\overline{\phantom{#1}}}{\widehat{\phantom{#1}}}{O}{c}{F}{F}{L}%
    }}%
  }{O}{c}{F}{F}{L}%
}

\newcommand{\bTMgn}{\simeqaccent{{\cat{M}}}_{g,n}}
\newcommand{\btMgn}{\triangleaccent{{\cat{M}}}_{g,n}}
\newcommand{\bMgn}{{\overline{\cat{M}}}_{g,n}}

\newcommand{\TM}{\widetilde{\cat{M}}}
\newcommand{\bTM}{\simeqaccent{{\cat{M}}}}

\newcommand{\TMgn}{\widetilde{{\cat{M}}}_{g,n}}

\newcommand{\TDelta}{\widetilde{\Delta}}

\newcommand{\VV}{\mathbb{V}}
\newcommand{\CC}{\mathbb{C}}

\newcommand{\QQ}{\mathbb{Q}}
\newcommand{\ZZ}{\mathbb{Z}}
\newcommand{\NN}{\mathbb{N}}
\newcommand{\PP}{\mathbb{P}}
\renewcommand{\SS}{\mathbb{S}}

\newcommand{\ee}{\mathscr{I}}

\newcommand{\EE}{\mathfrak{I}}
\newcommand{\Fk}{\mathfrak{F}}

\newcommand{\Oc}{\mathcal{O}}
\newcommand{\Lc}{\mathcal{L}}

\newcommand{\piM}{\Pi_1(\TM)}

\newcommand{\Sp}{\mathsf{Sp}}

\newcommand{\Spec}{\mathrm{Spec}}
\newcommand{\Vir}{\mathrm{Vir}}
\newcommand{\cbotimes}{\underset{\VV}{\otimes}}
\newcommand{\hlotimes}{\underset{{\text{\tiny HL}}}{\otimes}}
\newcommand{\MTA}{\mathfrak{A}}

\usepackage[all,cmtip]{xy}	
\xyoption{arrow}

\usepackage{wrapfig}

\usepackage{eso-pic}% The zero point of the coordinate systemis the lower left corner of the page (the default).

\title[Modular functors from conformal blocks of VOAs]{Modular functors from conformal blocks of rational vertex operator algebras}

\author{Chiara Damiolini}
\author{Lukas Woike}
\begin{document}

\begin{abstract}\noindent 
For a vertex operator algebra $V$, one may naturally define spaces of conformal blocks following a construction of Frenkel--Ben-Zvi generalized by Damiolini--Gibney--Tarasca. If $V$ is strongly rational, these spaces of conformal blocks form vector bundles over a suitable moduli space of algebraic curves. In this article, we establish, under the same assumptions, the widely expected topological result that the spaces of conformal blocks produce a modular functor, i.e.\ a modular algebra over an extension of the surface operad. This entails that the category $\Cc_V$ of admissible $V$-modules inherits from the topology of genus zero surfaces a ribbon Grothendieck--Verdier structure that leads even to the structure of a modular fusion category whose structure comes directly from the spaces of conformal blocks of $V$. As a direct consequence, we prove that the modular functor from conformal blocks extends to a three-dimensional topological field theory and comes with a description in terms of factorization homology.
\end{abstract}

\maketitle
%\tableofcontents
\vspace{-5mm}

\section{Introduction} 
Informally, a modular functor~\cite{Segal,ms89,turaev,tillmann,baki,jfcs} is a collection of  representations of extensions mapping class groups which satisfy appropriate compatibilities. 
More concisely, a modular functor can be described as a modular algebra over an appropriate extension of the surface operad~\cite{brochierwoike} as recalled below in \cref{def:MF}.
Translating this into the algebro-geometric context via the Riemann--Hilbert correspondence, one could describe a modular functor as a collection of (twisted) D-modules on an appropriate moduli space of curves, which again are required to satisfy appropriate conditions.  In this paper we show that the conformal blocks from a strongly rational vertex operator algebra (VOA) $V$ naturally define a modular functor. 

Let us recall that, given $V$-modules $X_1, \dots, X_n$, in \cite{FBZ:2004} Ben-Zvi and Frenkel construct the \textit{sheaf of coinvariants} $\VV_g(X_\bullet)$ over the moduli space $\TMgn$ of $n$-marked curves of genus $g$ (with non-zero tangent data). The sheaves of \textit{conformal blocks} are the dual sheaf. Furthermore, the conformal structure on $V$ (and on the modules $X_i$s) induces a projectively flat connection on $\VV_g(X_\bullet)$. This construction has been generalized in \cite{DGT1,DGT2} by Damiolini, Gibney and Tarasca to allow also curves with \textit{nodal singularities}. One obtains therefore a sheaf, which we still denote $\VV_g(X_\bullet)$, over the  moduli space $\bTMgn$ of $n$-marked possibly nodal curves of genus $g$ (with non-zero tangent data). When $V$ is strongly rational, in \cite{DGT2} it is further shown that $\VV_g(X_\bullet)$ is indeed a projectively flat vector bundle with logarithmic singularities (see \cref{sec:connection}). 
In this paper we show that these twisted D-modules define a modular functor:

\begin{customthm}{\ref{thm:VGisMF}} If $V$ is strongly rational, then sheaves of coinvariants define a modular functor.
\end{customthm}

The fact that the sheaves of coinvariants form a modular functor in the rational case was stated as an expectation by Ben-Zvi and Frenkel in \cite[Section~10.1.4]{FBZ:2004}. When $V=L_\ell(\mathfrak{g})$ (the simple affine VOA at positive integer level), \cref{thm:VGisMF} was proved in \cite{baki}. We note that their work is heavily based on the unpublished, but foundational manuscript of Beilinson, Feigin and Mazur \cite{BFM} which, besides considering the case of rational Virasoro VOAs, can also be seen as a precursor of \cite{baki}.

As a first remark, we note that the extension of $\VV_g(X_\bullet)$ to $\bTMgn$ (given in \cite{DGT1, DGT2}) is crucial to actually describe the compatibility conditions that a modular functor requires the bundles $\VV_g(X_\bullet)$ to satisfy. Informally, this follows from the fact that the topological operation of \textit{sewing} a surface together along some parametrized boundaries---which is one of the properties that a modular functor need to satisfy---in algebraic geometry can be translated into the two operations of degenerating a smooth curve to a nodal curve, and then smoothing it out to obtain another smooth curve (see also Figure~\ref{fig:teich}).

\smallskip

Before commenting on the proof of \cref{thm:VGisMF}, we  will first discuss two of its main consequences which we present in \cref{sec:final}.

\subsection{The topologically inherited structure on \texorpdfstring{$\Cc_V$}{CV}} By \cite[Theorem 7.17]{cyclic} every modular functor induces on its underlying category $\Cc$ naturally the structure of a \textit{ribbon Grothendieck--Verdier category}  (see \cref{sec:GVbackground} for details). Applying this to the modular functor obtained from coinvariants, we can therefore obtain a balanced braided monoidal product $\cbotimes$ on $\Cc_V$---the category of admissible $V$-modules---in an algebro-geometric fashion (see \cref{cor:tensor}). Moreover, by further inspection, and using the recent results of \cite{etingof.penneys:2024:rigidity}, we can show the following result:

\begin{customthm}{\ref{thm:modularfusion}} Coinvariants induce on $\Cc_V$ the structure of a modular fusion category.    
\end{customthm}

We observe that, together with~\cite{brochierwoike}, this result implies a factorization homology description of $\VV$ (Corollary~\ref{corconnected}). 

The fact that on the category of $V$-modules $\Cc_V$ one can define the structure of a modular fusion category has also been shown by Huang and Lepowsky. In fact, in a series of papers \cite{HL:I,HL:II,HL:III,HL:IV,huang:2005:verlinde,huang:rigidity}, the authors define a tensor product using analytic methods, denoted here $\hlotimes$ on $\Cc_V$, and they further show that this gives rise to a modular fusion category on $\Cc_V$. Our approach is independent of the work of Huang and Lepowsky. It crucially uses \cref{thm:VGisMF} which, as mentioned above, together with \cite[Theorem 7.17]{cyclic} naturally yields a ribbon Grothendieck-Verdier structure on $\Cc_V$. To verify that this actually is a modular fusion category one needs to prove \textit{rigidity}, which can be accomplished using \cite{etingof.penneys:2024:rigidity}. 
The non-degeneracy of the braiding can be proved by means of the factorization homology characterization of modular functors~\cite{brochierwoike}. 

It is natural to wonder whether in fact, the two structures are equivalent. As we discuss in \cref{sec:HL}, although one has isomorphisms $M \cbotimes N \cong M \hlotimes N$ for every $M, N \in \Cc_V$, it is not clear that these indeed realize an equivalence of modular categories. We expect that this is the case, but we leave the detailed analysis of this comparison to a future work and concentrate here on the implications of the geometrically constructed tensor structure in \cref{thm:modularfusion}.

\subsection{The 2dCFT/3dTFT correspondence}
One of the main ideas put forward by Witten in his seminal paper~\cite{Witten:1988hf} is a profound correspondence between two-dimensional conformal field theory and three-dimensional topological field theory.
The question of to what extent this correspondence is a rigorous mathematical result depends on the perspective:
If, as a starting point for the description of a conformal field theory, we choose a modular (fusion) category appearing in \cite{turaev} as an abstraction of the notion of modular data  in~\cite{mooreseiberg} (in that picture, a VOA is just a source for a modular category), then the work of Reshetikhin--Turaev~\cite{rt1,rt2,turaev} can be seen as a mathematically rigorous instance of the 2dCFT/3dTFT correspondence, at least at the level of state spaces:
In particular, this tells us how to build from a modular fusion category finite-dimensional vector spaces associated to surfaces that, in this particular approach, are called also spaces of conformal blocks~\cite[Section 2.4]{jfcs}. 

What is not covered in \cite{Witten:1988hf,rt1,rt2,turaev} is the treatment of \emph{full} conformal field theories, more precisely the construction of classification of consistent system of correlators for a conformal field theory whose monodromy data is given by a modular fusion category.
This was completely solved through the  \emph{FRS Theorem} of Fuchs--Runkel--Schweigert, partially with Fjelstad~\cite{frs1,frs2,frs3,frs4,ffrs,ffrsunique}. According to this result,  consistent systems of correlators arise from special symmetric Frobenius algebras in the modular fusion category. The result does not only ensure existence of the correlators; it is entirely constructive. In particular, partition functions can be can be calculated as link invariants~\cite{frs1}.

If however instead of modular categories, one puts the spaces of conformal blocks from \cite{FBZ:2004,DGT1,DGT2} front and center, Fuchs-Runkel-Schweigert point at a significant conceptual problem~\cite[Section~3.3]{csrcft}: The relation between the conformal blocks built from a modular fusion category via the Reshetikhin--Turaev construction (arising from the category of modules over a strongly rational VOA; those are the ones appearing in the FRS Theorem) and the spaces of conformal blocks that we consider here and arising from \cite{FBZ:2004,DGT1,DGT2} is not known.

Using \cref{thm:modularfusion}, we obtain the following result:

\begin{customthm}{\ref{thm:comparisonrt}}
Up to a contractible space of choices, the modular functor $\VV$ is the unique modular functor extending the modular fusion category $\Cc_V$ from genus zero to all surfaces. In particular, there is a canonical equivalence
\begin{align*}
    \VV \xrightarrow{\ \simeq \ } \Fk_{\Cc_V}
\end{align*}
of modular functors between $\VV$ and the Reshetikhin--Turaev type modular functor $\Fk_{\Cc_V}$ of the modular fusion category $\Cc_V$. This affords an extension of $\VV$ to a once-extended three-dimensional topological field theory.
\end{customthm}

To the best of our knowledge, this is the first 2dCFT/3dTFT correspondence genuinely native to the framework of rational VOAs, i.e.\ formulated for $\VV$. A comparison result in a similar spirit was achieved by Andersen and Ueno in \cite{andersenuenoinv} in the special case of affine Lie algebras. 

\Cref{thm:comparisonrt} has the important implication that the \emph{FRS Theorem}, i.e.\ the construction of correlation functions via the correspondence between two-dimensional conformal field theory and three-dimensional topological field theory, applies in a much more comprehensive way: It is actually applicable to the spaces of conformal blocks $\VV$ constructed directly from $V$.  
Since the FRS Theorem always guarantees non-trivial solutions for the correlation functions, it gives us solutions for the Knizhnik--Zamolodchikov equations of $V$. Actually, even more is true: By applying the FRS Theorem to $\Cc_V$ from \cref{thm:modularfusion}, we can describe \emph{all} systems of solutions compatible with gluing (see \cref{thm:correlators}).

It is important to note that this does not establish a comparison to what we would obtain by applying the FRS Theorem to the Huang--Lepowsky modular fusion category. 
But to profit from the FRS classification of correlators, we actually do not need this because of the extremely crucial fact that the FRS Theorem applies to \emph{any} modular fusion category, not just to modular fusion categories of a specific origin.

\subsection{On the
 proof of \texorpdfstring{Theorem~\ref{thm:VGisMF}}{the Main Theorem}}  The concept of modular functor that we consider (see \cref{def:MF}) is topological in nature, while conformal blocks from VOAs, as defined in \cite{FBZ:2004}, live in the algebro-geometric setting. Inspired by \cite{baki,deshpande.mukhopadhyay:2019}, in \cref{sec:MFbackground}, we describe the dictionary that is necessary to go from one perspective to the other. Let $V$ be a strongly rational VOA and, as above, let $\Cc_V$ be the category of admissible $V$-modules, which is equipped with an equivalence $\Cc_V^\op \to \Cc_V$ obtained by assigning to a module $M$ its contragredient $M'$. A modular functor on $\Cc_V$, should be given by a compatible way of assigning, to every stable pair $(g,n)$, and every $X_1, \dots, X_n \in \Cc_V$, a vector bundle on $\bTMgn$ with a projectively flat connection (with logarithmic singularities along the boundary $\TDelta_{g,n}$).  \Cref{thm:VGisMF} essentially boils down to the claim that the assignment
\[ X_1, \dots, X_n  \mapsto \VV_{g}(X_1, \dots, X_n)
\] satisfies explicit compatibilities corresponding to \textit{natural} operations between marked curves.

Let us unpack what we mean by this. 
The stacks $\bTMgn$ are connected to each other by \textit{tautological maps} which arise from forgetting some of the marked points, or gluing together curves along chosen marked points. These operations arise from the topological counterpart of \textit{capping boundary components} and \textit{sewing together} two boundary components respectively  (see also Figure~\ref{fig:teich}). 

Denote by $\xi_{n+1} \colon \bTM_{g,n+1} \to \bTM_{g,n}$ the map that forgets the datum of the $n+1$-st point, then \cite[Propagation of Vacua]{codogni:POV,DGT1} provides an isomorphism
\begin{equation}\label{eq:povintro} \xi_{n+1}^* \VV_g(X_1, \dots, X_n) \cong \VV_g(X_1, \dots, X_n, V)\end{equation} of sheaves on $\bTM_{g,n+1}$. However, this is not sufficient: Both sides of \eqref{eq:povintro} are twisted D-modules (i.e.\ vector bundles with a projectively flat connection) and one of the requirements that $\VV$ defines a modular functor is that the isomorphism \eqref{eq:povintro} is actually an isomorphism preserving this extra structure. We show this is the case in \cref{prop:pov} by a direct inspection of the definition of the projectively flat connection given in \cite{FBZ:2004,DGT1}. We note that, for this compatibility to hold, it is actually not necessary to assume that $V$ is rational.

The more involved requirement that $\VV$ has to satisfy, and where we heavily make use of the rationality of $V$, concerns the compatibility with the gluing map, analyzed in \cref{sec:gluing}. Describing this compatibility carefully in algebraic terms is a little more difficult and requires some notation. Given an element in $\bTMgn$, i.e.\ a curve with marked points and tangent vectors, one can naturally obtain a new curve by gluing together two distinct points. One can show that this operation actually gives rise to a map 
\[ \Sp \colon \{ \text{twisted D-modules over } \bTM_{g,n}\} \to \{ \text{twisted D-modules over } \bTM_{g-1,n+2}\}.
\] The required compatibility that $\VV$ has to satisfy is the datum of an \textit{isomorphism of twisted D-modules} 
\begin{equation} \label{eq:excisionintro} \Sp \, \VV_g(X_1, \dots, X_n) \cong \VV_{g-1}(\Delta, X_1, \dots, X_n)
\end{equation} over $\bTM_{g-1,n+2}$, where $\Delta \in \Cc_V \boxtimes \Cc_V$ is the \textit{coevaluation element} given by the contragredient equivalence $\Cc^\op \to \Cc$. We first construct such an isomorphism using the \cite[Sewing Theorem]{DGT2} and we then show that indeed this isomorphism is compatible with the projectively flat connection. In both these steps, a crucial ingredient is the rationality of $V$: in fact this implies that 	we can identify $\Delta$ with the finite sum $\bigoplus_{S \text{ simple }} S \otimes S'$, which makes it possible not only to construct an explicit map \eqref{eq:excisionintro}, but also to control its behavior with respect to the projective connection.

\subsection{Further directions} We conclude the introduction mentioning one natural question which we also briefly discuss in the last section of \cref{sec:final}. As pointed out, \cref{thm:VGisMF} assumes that $V$ is a strongly rational VOA, and the rationality condition plays a crucial role in the proof of \cref{thm:VGisMF}. However, imposing that $V$ is rational is very restrictive. For instance, in \cite{HLZ}, Huang, Lepowsky and Zhang show that even without the rationality assumption, the category of $V$-modules $\Cc_V$ can be endowed with a monoidal structure (which turns out to be Grothendieck--Verdier~\cite{alsw}). It is natural then to ask if sheaves of coinvariants induce a monoidal structure on $\Cc_V$ also without the rationality condition (see \eqref{qs:CBareE2algebras}).

\subsection{Plan of the paper} In \cref{sec:CBrecall} we give a quick overview of the definition of conformal blocks associated to VOAs, referring the reader to \cite{FBZ:2004,DGT1,DGT2} for detail. In \cref{sec:MFbackground} we review the notion of a modular functor from an operadic perspective and we explain how this can be described in algebro-geometric terms. \Cref{thm:VGisMF} is stated and proved in \cref{sec:CBareMF}. Finally, in \cref{sec:final} we describe the main consequences of \cref{thm:VGisMF}, such as \cref{thm:modularfusion} and \cref{thm:correlators}, and discuss future research directions and open questions.

\subsection{Acknowledgments.} We would like to warmly thank Christoph Schweigert for the many discussions and comments related to this project. Most of this project was carried out while CD was a visitor of the University of Hamburg as a Mercator Fellow: she thanks I. Runkel, C. Schweigert, and J. Teschner. Thanks also to Adrien Brochier, David Ben-Zvi, Jürgen Fuchs, Alexander Kirillov, Swarnava Mukhopadhyay and Lukas Müller for conversations and comments related to this project.

CD is supported by NSF DMS 2401420. 
LW gratefully acknowledges support by the ANR project CPJ n°ANR-22-CPJ1-0001-01 at the Institut de Mathématiques de Bourgogne (IMB). The IMB receives support from the EIPHI Graduate School (contract ANR-17-EURE-0002).

\section{Vertex operator algebras and associated conformal blocks} \label{sec:VOAs-CB}

Throughout the paper, $V$ will denote a vertex operator algebra (VOA for short) which is assumed to be $\NN$-graded, of CFT-type, $C_2$-cofinite, rational and self-contragredient. This in particular implies that $V$ is simple. Phrased differently, $V$ is a \textit{strongly rational} VOA. Throughout, by $\Cc_V$ we denote the category of admissible $\NN$-graded $V$-modules. We refer to \cite{FHL,Lepowsky.Li} as well as \cite{DGT2,DGT3} for more background on VOAs, $V$-modules and on the conditions imposed here. 

\subsection{Coinvariants and conformal blocks} \label{sec:CBrecall} Generalizing the constructions of Ben-Zvi--Frenkel and Nagatomo--Tsuchiya \cite{FBZ:2004, NT},  Damiolini, Gibney and Tarasca showed in \cite{DGT1, DGT2} that, starting with $n$ objects $X_1, \dots, X_n\in \Cc_V$, one can naturally define a vector bundle $\VV_g(X_1, \dots, X_n)$---called the \textit{sheaf of coinvariants}---on $\bTMgn$, the moduli stack parametrizing stable genus $g$ curves marked by $n$ points at each of which a non-zero tangent vector is fixed (this stack is denoted $\overline{\mathcal{J}}^{1,\times}_{g,n}$ in \cite{DGT1}).  This construction is obtained by descent from $\btMgn$, the moduli stack parametrizing stable genus $g$ curves marked by $n$ points and local coordinates.\footnote{We note that in \cite{DGT1} and \cite{DGT2}, two a priori different sheaves of coinvariants are defined. However in \cite{DGK2} it is shown that actually the two constructions are equivalent.}

We denote the fiber of $\VV_g(X_1, \dots, X_n)$ at the point $(C,P_\bullet, \tau_\bullet)$ of $\bTMgn$ by $\VV(X_\bullet)_{[C,P_\bullet, \tau_\bullet]}$ or simply by $\VV(X_\bullet)_{[C]}$ if $(P_\bullet, \tau_\bullet)$ are understood. We call these spaces the \emph{spaces of coinvariants} associated to $X_\bullet$ and to the point $(C,P_\bullet, \tau_\bullet)$ of $\bTMgn$. These spaces are defined as quotients of $X_1 \otimes \dots \otimes X_n$ by the action of a Lie algebra $\Lc_{C\setminus P^\bullet}(V)$ which takes into consideration both the geometry of $C$ and the action of $V$ on the modules $X_i$ (see \cite[Section 3]{DGT2} for further details). Spaces of \emph{conformal blocks} are naturally defined to be the dual of spaces of coinvariants and so they can be interpreted as spaces of functions $X_1 \otimes \dots \otimes X_n \to \CC$ which satisfy some constraints specified by $\Lc_{C\setminus P^\bullet}(V)$. 

\begin{remark} \begin{enumerate}[label=(\alph*)] 
\item When $V$ is rational and $C_2$-cofinite and the modules $X_\bullet$ are simple, then one can prove that the bundle $\VV_g(X_1, \dots, X_n)$ is independent of the non-zero tangent vectors and thus yields a vector bundle on $\bMgn$ (see \cite{DGT2}). For the purpose of this paper, however, we will always work over $\bTMgn$. 
\item We note that in the construction of the sheaves $\VV_g(X_1, \dots, X_n)$, one does not need that $V$ is $C_2$-cofinite or rational. However, without both these assumptions, we do not know whether $\VV_g(X_1, \dots, X_n)$ is locally free of finite rank. Assuming that $V$ is $C_2$-cofinite, but not necessarily rational, in \cite{DGK1} it is shown that the sheaves $\VV_g(X_1, \dots, X_n)$ are coherent. In \cite{DGK2} a sufficient condition to ensure that they are locally free is provided. \end{enumerate}
\end{remark}

\subsection{Projectively flat connection} \label{sec:connection}  By extending the work of \cite{FBZ:2004}, in \cite[Section 7]{DGT1} it is shown that the bundle $\VV_g(X_1, \dots, X_n)$ admits a projectively flat connection with logarithmic singularities along the boundary divisor $\TDelta_{g,n}$ parametrizing singular curves. This can be interpreted as a generalization of the Knizhnik--Zamolodchikov connection.

We set up some notation to explain this a little more carefully. If $\TMgn$ denotes the moduli stack parametrizing \textit{smooth} genus $g$ curves marked by $n$ points and with non-zero tangent vectors at each marked point, then $\TDelta_{g,n} \coloneqq \bTMgn\setminus \TMgn$. To every line bundle $\Lc$ on $\bTMgn$ and to any $\alpha \in \CC$, one naturally associates a central extension 
\begin{equation}\label{eq:atiyah} 0 \longrightarrow \Oc_{\bTMgn} \longrightarrow  \alpha \Ac_{\Lc} \longrightarrow \Tc_{g,n} \longrightarrow 0.
\end{equation} of the tangent bundle $\Tc_{g,n}$ of $\bTMgn$. (Here, and throughout, for $X$ a scheme (or a stack), $\Oc_X$ denotes the sheaf of rational functions on $X$.)
In more detail, when $\alpha \in \ZZ$ the sheaf $\alpha \Ac_{\Lc}$ can be interpreted as the sheaf of first order differential operators on the bundle $\Lc^{\otimes \alpha}$. This notion can naturally be extended to every $\alpha \in \CC$ (see \cite{Tsuchimoto}). By restricting this sequence to $\Tc_{g,n}(-\TDelta_{g,n}) \subset \Tc_{g,n}$,  the sheaf of tangent vectors to $\bTMgn$ which are tangent to the boundary divisor $\TDelta_{g,n}$, one obtains the sheaf of Lie algebras $\alpha \Ac_{\Lc}(-\TDelta_{g,n})$ and the exact sequence
\begin{equation}\label{eq:atiyahLog} 0 \longrightarrow \Oc_{\bTMgn} \longrightarrow  \alpha \Ac_{\Lc}(-\TDelta_{g,n}) \longrightarrow \Tc_{g,n}(-\TDelta_{g,n})\longrightarrow 0
\end{equation} analogous to \eqref{eq:atiyah}.
In \cite{DGT1} it is proved that indeed $\VV_g(X_1, \dots, X_n)$ is equipped with an action of 
\[ \dfrac{c}{2}\Ac_{\Lambda}(-\TDelta_{g,n}),
\]  where $c$ is the central charge of the VOA $V$ and where $\Lambda$ is the Hodge line bundle on $\bTMgn$. This connection arises from the fact that, by definition, each $V$-module is equipped with an action of the Virasoro algebra with central charge $c$. More details are given in \cref{subsec:pov}. We remark here that the existence of this connection and its identification with the action of this Atiyah algebra does not require the rationality of $C_2$-cofiniteness of $V$. We refer the reader to \cite[Section 7.1]{DGT1} and references therein for more details about Atiyah algebras. In this paper, we denote the category of vector bundles on $\bTMgn$ equipped with an action of $\alpha\Ac_{\Lambda}(-\TDelta_{g,n})$ by $(\alpha\Ac_{\Lambda}\text{-mod})(\bTMgn)$.

\subsection{Propagation of Vacua, Factorization, and Sewing} As illustrated in \cite{DGT2} it is shown that the sheaves of coinvariants satisfy natural compatibility conditions which resemble---but are a priori weaker---than those satisfied by a modular functor. In this paper we will show that, in fact, sheaves of coinvariants define a modular functor (see \cref{thm:VGisMF}). To do so, we will heavily rely on the properties known as Propagation of Vacua (see \cite[Theorem 3.6]{codogni:POV} and \cite[Theorem 6.2]{DGT1}) and the Sewing Theorem (\cite[Theorem 8.5.1]{DGT2}), which in turn can be interpreted as an enhancement of the Factorization Theorem (\cite[Theorem 7.0.1]{DGT2}). Before doing this, we will recall the notion of modular functor and return to the statement and proof of \cref{thm:VGisMF} in \cref{sec:CBareMF}.

\section{Modular functors} \label{sec:MFbackground}
Using the notion of a modular functor that builds, among other works, on \cite{ms89, turaev, tillmann, baki, jfcs}, one may describe consistent systems of representations of extensions of mapping class groups. 
In this text, we will use the definition from \cite{brochierwoike} that we recall in this section.

\subsection{Modular operads}\label{sec:modoperad} We will define modular functors via modular operads and their algebras that were introduced  by Getzler and Kapranov in \cite{gkmod}. 
More precisely, we will need modular operads with values in the symmetric monoidal bicategory $\Cat$ of small categories. 
Intuitively, a $\Cat$-valued modular operad associates to a natural number $n\ge 0$, a category of operations of total arity $n$, in such a way that (a long list of) natural compatibilities are satisfied.  A way to effectively encode these compatibilities is to think about $n$ as a graph having only one vertex and $n$ many legs.
One can then build a category out of such graphs and their disjoint unions. This description based on categories of graphs was given by Costello~\cite{costello} and adapted to a bicategorical framework in \cite[Section~2]{cyclic}.

We define  $\Graphs$ as the category having as objects finite disjoint unions of finite sets (which, as we see next, can be identified with appropriate graphs). 
Maps in this category are given by graphs as we explain now: 
Recall that a {graph} is a tuple \[\Gamma=(V,H, i\colon H \to H, s \colon H \to V)\] where $V$ and $H$ are finite sets of vertices and half edges, respectively, and $i \colon H\to H$ is an involution. Intuitively, the map $s$ sends a half edge to its source vertex, while $i$ maps a half edge to the half edges that it is glued to. The orbits of $i$ are called \emph{edges} of the graph. An orbit of size one is called a \emph{leg} while an orbit of size two is called an \emph{internal edge}. 

A graph with one vertex and finitely many legs attached to it is called a \emph{corolla}. We will therefore identify the objects of $\Graphs$ with finite disjoint unions of corollas. 
\begin{wrapfigure}{r}{66mm}
\centering
\begin{tikzpicture}[scale=0.27]
	\begin{pgfonlayer}{nodelayer}
		\node [circle,fill,inner sep=1.3pt] (0) at (-4, 4) {};
		\node [circle,fill,inner sep=1.3pt] (1) at (-4, 0) {};
		\node [style=none] (2) at (-4, 7) {};
		\node [style=none] (3) at (-6, -2) {};
		\node [style=none] (4) at (-2, -2) {};
		\node [style=none] (5) at (-7, 6) {};
		\node [style=none] (6) at (-4, -3) {$\Gamma$};
		\node [circle,fill,inner sep=1.3pt] (7) at (2.75, 4) {};
		\node [circle,fill,inner sep=1.3pt] (8) at (2.75, 0) {};
		\node [style=none] (9) at (2.75, 7) {};
		\node [style=none] (10) at (0.75, -2) {};
		\node [style=none] (11) at (4.75, -2) {};
		\node [style=none] (12) at (1.5, 6) {};
		\node [style=none] (13) at (0, 4.75) {};
		\node [style=none] (14) at (2.75, 2.5) {};
		\node [style=none] (15) at (2.75, 1.5) {};
		\node [style=none] (16) at (4, 2.5) {};
		\node [style=none] (17) at (4, 1.5) {};
		\node [style=none] (18) at (2.75, -3) {$\nu(\Gamma)$};
		\node [circle,fill,inner sep=1.3pt] (19) at (9.5, 2.5) {};
		\node [circle,fill,inner sep=1.3pt] (20) at (9.5, 2.5) {};
		\node [style=none] (21) at (9.5, 7) {};
		\node [style=none] (22) at (7.5, -2) {};
		\node [style=none] (23) at (11.5, -2) {};
		\node [style=none] (24) at (9.5, -3) {$\pi_0(\Gamma)$};
	\end{pgfonlayer}
	\begin{pgfonlayer}{edgelayer}
		\draw [bend right=90, looseness=1.25] (1) to (0);
		\draw (0) to (1);
		\draw (0) to (2.center);
		\draw [in=405, out=105] (0) to (5.center);
		\draw [in=165, out=-135, looseness=1.25] (5.center) to (0);
		\draw (1) to (3.center);
		\draw (1) to (4.center);
		\draw (7) to (9.center);
		\draw (8) to (10.center);
		\draw (8) to (11.center);
		\draw (7) to (12.center);
		\draw (7) to (13.center);
		\draw (7) to (14.center);
		\draw (7) to (16.center);
		\draw (8) to (15.center);
		\draw (8) to (17.center);
		\draw (19) to (21.center);
		\draw (20) to (22.center);
		\draw (20) to (23.center);
	\end{pgfonlayer}
\end{tikzpicture} 
%\vspace*{0.2cm}
\label{figgraphs}
\end{wrapfigure}  

    In order to define the morphisms of $\Graphs$, we first introduce two operations that associate to a graph, a disjoint union of corollas. Given a graph $\Gamma$, we define 
\begin{itemize} 
    \item $\nu(\Gamma)$ to be the graph obtained by cutting $\Gamma$ at all edges;
    \item $\pi_0(\Gamma)$ to be the graph obtained by contracting all internal edges of $\Gamma$.
\end{itemize} 

A morphism between two corollas $T \to T'$ is given by an equivalence class of triples $(\Gamma, \varphi_0, \varphi_1)$, where $\Gamma$ is a graph and where $\varphi_i$ are isomorphisms 
$\varphi_0 \colon T \cong \nu(\Gamma)$ and $\varphi_1 \colon T' \cong \pi_0(\Gamma)$. We refer to \cite{cyclic} for more details on the equivalence relation and on the composition of such maps.

Disjoint union endows $\Graphs$ with a symmetric monoidal structure. We will see $\Graphs$ as a symmetric monoidal bicategory whose only $2$-morphisms are identities.

\begin{exx} \label{exxsn} Denote, for every $n \ge 0$, the corolla $T_{n-1}$ with $n$ legs. One deduces from the definition of $\Graphs$ that the permutation group $S_n$ on $n$ letters acts on $T_{n-1}$.
\end{exx}

We now have the language to define a {modular operad} which, in this paper, we will always assume to be valued in $\Cat$, the symmetric monoidal bicategory having as objects (essentially small) categories, functors as 1-morphisms and natural transformations as 2-morphisms.

\begin{definition} \label{def:modoperad} A \emph{($\Cat$-valued) modular operad} is a symmetric monoidal functor
\begin{equation*}
\cat{O} \colon \Graphs \to \Cat,
\end{equation*}
where, as mentioned above, $\Graphs$ and $\Cat$ are both interpreted as symmetric monoidal bicategories, that comes equipped with an operation $1_\cat{O}\in \cat{O}(T_1)$ of total arity two, called \emph{operadic identity}, that behaves neutrally with respect to operadic composition up to coherent isomorphism
and comes equipped with the structure of a homotopy fixed point with respect to the homotopy involution on $\cat{O}(T_1)$ as specified in \cite[Definition~2.3]{cyclic}.  \end{definition}

\begin{remark}
In particular, $\cat{O}$ assigns to every corolla $T$ a category $\cat{O}(T)$.
If $T_n$ is a corolla with $n+1$ legs for some $n\ge -1$, then $\cat{O}(T_n)$ is the category of $n$-ary operations, i.e.\ operations with $n$ inputs and one output. In the context of modular operads, the distinction of input versus output is not made however; both are treated on the same footing.
Therefore, one can refer to $\cat{O}(T_n)$ as the category of operations of total arity $n+1$.
To every graph, seen as a morphism $\Gamma:T\to\ T'$ in $\Graphs$, the modular operad $\cat{O}$ associates a functor $\cat{O}(\Gamma) \colon \cat{O}(T)\to \cat{O}(T')$. Since $\cat{O}$ is required to be a symmetric monoidal functor, this assignment has to satisfy appropriate compatibility conditions that we will not spell out here. We should also point out that a symmetric monoidal functor between symmetric monoidal bicategories is always to be understood in a weak sense, i.e.\ up to coherent isomorphism.
The necessary framework for this is set up in \cite[Chapter~2]{schommerpries}, see the explanations after \cite[Definition~2.2]{cyclic} for more details and the connection to more traditional descriptions of (modular) operads.  
\end{remark}

\subsection{The modular surface operad} The main example of a modular operad that informed the invention of the concept in \cite{gkmod} in the first place is the \emph{surface operad} \[\Surf \colon \Graphs \to \Cat. \] For the category-valued version of this operad needed in this section, we refer e.g.\ to \cite[Section~7.3]{cyclic}. 
Let us outline the definition here: For a corolla $T$, the category $\Surf(T)$ is the groupoid whose objects are connected surfaces $\Sigma$ (surface means for us always oriented, smooth, two-dimensional manifold, possibly with boundary) equipped with an orientation-preserving diffeomorphism $\sqcup_{\Legs(T)} \SS^1 \ra{\cong} \partial \Sigma$ that is referred to as \emph{boundary parametrization}.  Morphisms in $\Surf(T)$ are isotopy classes of diffeomorphisms preserving the orientation and the boundary parametrization. In particular, the automorphism group of a surface $\Sigma \in \Surf(T)$ is the pure mapping class group of $\Sigma$; we refer to~\cite{farbmargalit} for a textbook introduction. 
The operadic composition is given by gluing along boundary circles.

\begin{remark} \label{rmk:operad vs tower} The reader familiar with \cite{baki} will note a certain similarity between the notion of a modular operad (with values in $\Gpd$ rather than $\Cat$), and that of a \emph{tower of groupoids}. 
Although not all the required compatibilities are spelled out in \cite{baki}, these two concepts \textit{morally} correspond to one another. In particular, many arguments of \cite{baki} phrased for towers of groupoids can be adapted and applied to modular operads. Under this correspondence, the operad $\Surf$ corresponds to the tower of groupoids $\piM$ given in \cite[Section 5.6]{baki}.  
Note however the warning in Remark~\ref{remmfdef} that applies once we pass to representations of these topological structures.
\end{remark}

\begin{remark} \label{rmk:Surf=TM} We note that for every corolla $T$, one has a decomposition \[\Surf(T)=\bigsqcup_{g \in \NN} \Surf_g(T),\] 
where $\Surf_g(T)\subset \Surf(T)$ is the full subgroupoid spanned by surfaces of genus $g$. (Note however that $\Surf_g$ is \emph{not} a modular operad!) For the  corolla $T_{n-1}$ with $n$ legs, denote $\Surf_g(T_{n-1})$ by $\Surf_{g,n}$. Whenever $(g,n)$ is a \textit{stable pair}, i.e.\ different from $(0,0)$, $(0,1)$ and $(1,0)$, then $\Surf_{g,n}$ is equivalent to the groupoid $\Pi_1(\TMgn)$. This is the fundamental groupoid of $\TMgn$, the stack whose objects are smooth and projective curves of genus $g$ with $n$ marked points and non-zero tangent vectors (see \cite[Theorem 6.1.13]{baki} and also \cite[Theorem 14.4]{deshpande.mukhopadhyay:2019}). Denote by $\piM$ the collection of groupoids $\Pi_1(\TMgn)$ varying over stable $(g,n)$. This corresponds to an appropriate restriction of the operad $\Surf$ which we call $\Surf^{\mathrm{stab}}$, in analogy with the notation $\Teich^{\mathrm{stab}}$ used in \cite{baki}
(but note that $\Surf^{\mathrm{stab}}$ is not a modular operad).

As detailed in \cite[Sections 6.2]{baki}, one naturally associates to every graph $\Gamma \colon T \to T'$, a \textit{gluing} functor $\piM(\Gamma) \colon  \piM(T) \to \piM(T')$. The definition of this morphism uses the fact that, to $\Gamma$ and to a curve belonging to $\piM(T)$, one can naturally associate a nodal curve whose dual graph is $\Gamma$. To such a curve (and remembering the data of the tangent vectors), one can naturally associate a smooth curve which will then be an object of $\piM(T')$. 

This procedure does not work if $(g,n)$ is not stable. However, there is only one piece of datum that is missing if one uses $\Surf^{\mathrm{stab}}$---or equivalently $\piM$---rather than whole operad $\Surf$: The operation of filling punctures. 
Let $T_0$ be the corolla with one leg. The groupoid $\Surf_0(T_0)$ is actually trivial: 
Its unique object up to isomorphism
is the unit disk $\Dd$ and the identity is its only automorphism. Given an object $T$ of $\Graphs$ and one of its legs $\ell$,
we can consider the object $T(\ell)$ of $\Graphs$ obtained by removing the leg $\ell$ from $T$. To this operation, there is a natural map of groupoids $\TM_{f_\ell} \colon \piM(T) \to \piM(T(\ell))$ which contracts the $\ell$-th point.

We also observe that the map $\Gamma_\ell \colon T_0 \sqcup T \to T(\ell)$, given by the graph obtained gluing the leg of $T_0$ with the leg $\ell$ of $T$ to form an edge, induces the analogue map between surfaces. Namely, for every $\Sigma \in \Surf(T)$, one obtains that  $\Sigma_\ell=\Surf(\Gamma_\ell)(\Dd \sqcup \Sigma)$ is given by $\Sigma$ by \textit{capping} the $\ell$-th boundary component with the disk $\Dd$. 

This will allow us to treat the modular operad $\Surf$ and the tower of groupoids $\piM$ as essentially equivalent concepts, at least for arity-wise arguments, see  Figure~\ref{fig:teich} for a visualization.

\begin{figure*}[h!]
\centering
\includegraphics[width=\textwidth]{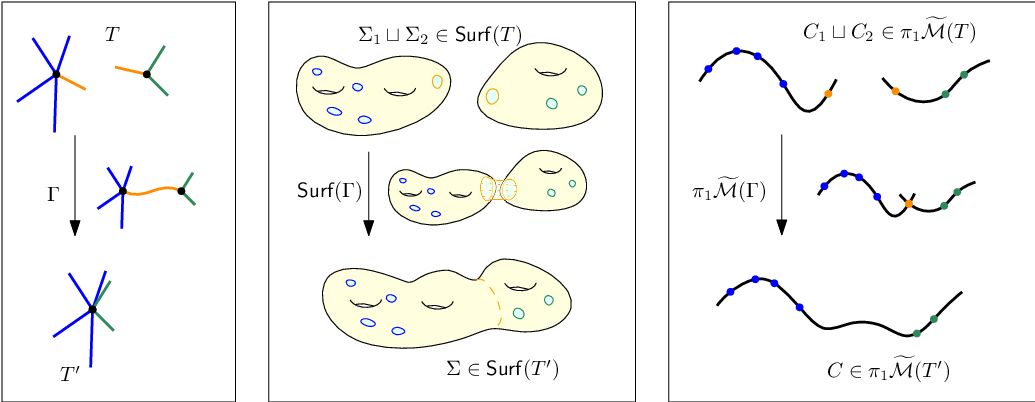}
\caption{Illustration of $\Surf$ and of $\piM$.}
\label{fig:teich}
\end{figure*}

\end{remark}

\subsection{Modular algebras over modular operads\label{secmodularalg}} Given a modular operad $\cat{O}$---such as $\Surf$---we can define algebras over it with values in any symmetric monoidal bicategory $\cat{S}$.
The case relevant for this article is $\cat{S}=\Rexf$, the symmetric monoidal bicategory of \emph{finite categories}~\cite{etingofostrik} over $\CC$. The objects of $\Rexf$ are linear abelian categories with finite-dimensional morphism spaces, finitely many simple objects up to isomorphism, enough projective objects and finite length for every object, 1-morphisms are right exact functors and 2-morphisms are linear natural transformations. The Deligne product $\boxtimes$ is the symmetric monoidal product,
and the category $\vect$ of finite-dimensional $\CC$-vector spaces is the monoidal unit. 

For $\Cc\in\Rexf$, consider a \emph{non-degenerate symmetric pairing} $\kappa \colon \Cc\boxtimes\Cc\to \vect$. 
Here the \emph{non-degeneracy} means that  $\kappa$ exhibits $\Cc$ as its own dual in the homotopy category of $\Rexf$; in more detail, there is a right exact functor $\Delta \colon \vect\to\Cc\boxtimes\Cc$ satisfying appropriate compatibilities with $\kappa$.
The \emph{symmetry} of $\kappa$ means that $\kappa$ is equipped with the structure of a homotopy fixed point structure for the $\mathbb{Z}_2$-action on maps $\Cc\boxtimes\Cc\to\vect$ coming from the symmetric braiding of $\boxtimes$.

As shown in \cite[Proposition 2.12]{cyclic}, associated to $(\Cc, \kappa)$, one naturally defines the modular operad
$\EndAk$ by sending a corolla $T$ with $n$ legs to the category
\[\End^\Cc_\kappa(T):=\Rexf(\Cc^{\boxtimes n},\vect).\]
The pairing $\kappa$, or rather its associated map $\Delta \colon \vect \to \Cc \boxtimes \Cc$, is used to define the functor $\EndAk(\Gamma)$, so that one obtains indeed a modular operad.
We can further interpret $\Delta$ as an object of $\Cc\boxtimes\Cc$, the so-called \emph{coevaluation object}, which coincides with the \emph{end} $\Delta = \int_{X\in\Cc} DX \boxtimes X$ (see \cite[Section~2.2]{fss} for an introduction to (co)ends in finite linear categories).

\begin{definition} A \emph{modular $\cat{O}$-algebra valued in $\Rexf$} is an object $\Cc\in\Rexf$ equipped with a non-degenerate symmetric pairing $\kappa$ and a morphism \[\Fk\colon \cat{O}\to \End_\kappa^\Cc\] of modular operads, i.e.\ a symmetric monoidal natural transformation of symmetric monoidal functors $\Graphs\to\Cat$ preserving operadic identities up to coherent isomorphism. \end{definition}

For the more detailed version of this definition, we refer to \cite[Section~2.4]{cyclic}. Let us partly unpack  this in the case $\cat{O}=\Surf$. For $T_{n-1}$, the corolla with $n$ legs and $\Sigma \in \Surf(T_{n-1})$, a modular $\Surf$-algebra
$\Fk\colon \cat{O}\to \End_\kappa^\Cc$ assigns a right exact functor \[\Fk(\Sigma;-)\colon \Cc^{\boxtimes n}\to\vect\] carrying a representation of the mapping class group $\Map(\Sigma)$ through natural automorphisms. These assignments need to satisfy appropriate compatibilities. We are going to focus on two of them.

For $X_1,\dots,X_n \in \Cc$ (to be thought of as boundary labels for the $n$ boundary components of $\Sigma$), the vector space $\Fk(\Sigma;X_1,\dots,X_n)$ is a representation of the mapping class group of $\Sigma$. Suppose that $\Sigma'$ is obtained by gluing two boundary circles of another surface $\Sigma$ together, then the modular $\Surf$-algebra structure gives us an isomorphism
\begin{equation}
\Fk(\Sigma' ; -)\cong \Fk(\Sigma;\Delta,-), \label{eqnexcision}
\end{equation}
where $\Delta \in \Cc\boxtimes\Cc$ occupies the two slots affected by the gluing, without loss of generality the two first ones. 
This is sometimes referred to as \emph{excision} or \emph{factorization}. The isomorphism is $\Map(\Sigma)$-equivariant, with the $\Map(\Sigma)$-action on
$\Fk(\Sigma' ; -)$ obtained via restriction along the group morphism $\Map(\Sigma)\to\Map(\Sigma')$. 
If $\Cc$ is semisimple --- let us denote the finitely many simple objects by $X_i$---then $\Delta \cong \bigoplus_i DX_i \boxtimes X_i$, which reduces~\eqref{eqnexcision} to
\begin{equation}
    \Fk(\Sigma' ; -)\cong \bigoplus_i \Fk(\Sigma;DX_i,X_i,-) \ . \label{eqnexcisionSS}
\end{equation}

A special case of the above occurs when $\Sigma$ is the disjoint union of the disk $\Dd$ and of another surface $\Sigma''$. Suppose that $\Sigma'$ is obtained by gluing the boundary circle of $\Dd$ with one boundary circle of $\Sigma''$. In this special case, and assuming that $\Cc$ is semisimple,  \eqref{eqnexcisionSS} becomes
\[     \Fk(\Sigma' ; -)\cong \bigoplus_i \Fk(\Dd, DX_i) \otimes \Fk(\Sigma'';X_i,-) \ .
\] In particular, by setting $E \coloneqq \bigoplus_i\Fk(\Dd, DX_i) \otimes  X_i \in \Cc$,  we obtain
\begin{equation} \Fk(\Sigma' ; -)\cong \Fk(\Sigma''; E, -). \label{eq:pov}
\end{equation} Note that, by construction, $\Sigma''$ arises from $\Sigma'$ by adding an extra boundary component. 
Therefore (an iterative application of) \eqref{eq:pov} tells us that the value of $\Fk$ does not change if the surface it depends on is modified adding more boundary components to which the element $E$ is attached.

\begin{remark} \label{rmk:list} We list a series of observations that will be used in \cref{sec:CBareMF}.\begin{enumerate}[label={(\alph*)}]
\item \label{rmk:povextra} We observe that, when considering $\Pi_1(\bTM)$ in place of $\Surf$ to define the concept of modular algebra, then condition \eqref{eq:pov} needs to be added separately. In fact, as we already discussed, in $\Pi_1(\bTM)$ only stable pairs $(g,n)$ are allowed. Therefore the compatibilities with the operation of capping a puncture, translated here into forgetting one of the marked points, need to be imposed as a separate requirement. As we will see in \cref{subsec:pov}, this is related to the property called \textit{propagation of vacua}.
\item As in \cref{rmk:operad vs tower}, one might compare the notion of modular $\cat{O}$-algebra valued in $\Rexf$ with the one of \textit{representation} of a tower of groupoids \cite[Definition 5.6.12]{baki}. Although the two notions slightly differ (e.g., we only allow right exact functors, while no such restriction is in place in \cite{baki}; moreover, \cite{baki} does seem
to consider many coherence conditions one needs to bicategorical operadic algebras), one can deduce properties of $\cat{O}$-algebras valued in $\Rexf$ using the properties of representations of tower of groupoids described in \cite{baki} and adapt many arguments of \cite{baki} to our situation (at least for arity-wise arguments), see however Remark~\ref{remmfdef}.
\item \label{rmk:Dmod}
    Using the relations between $\Surf$ and $\piM$ described in \cref{rmk:Surf=TM}, together with the Riemann--Hilbert correspondence, we can give an alternative description of a $\Surf$-algebra in the algebro-geometric setting (see also \cite[Theorem 6.4.2]{baki}). For every stable pairs $(g,n)$, and for every $X_1, \dots, X_n \in \Cc$, a $\Surf$-algebra $\Fk$ naturally assigns a vector bundle of finite rank $\Fk_{g}(X_1, \dots, X_n)$ on $\bTMgn$ with a flat connection having logarithmic singularities along the boundary $\TDelta_{g,n}=\bTMgn \setminus \TMgn$. This system of logarithmic D-modules (varying $g$, $n$ and $X_i$), must satisfy a series of compatibilities analogous to those described in \eqref{eqnexcision} and \eqref{eq:pov}.
    \end{enumerate} 
\end{remark}

\subsection{Modular functors}
In an overly simplified way, modular functors (with values in $\Rexf$)
could be thought of as $\Rexf$-valued modular $\Surf$-algebras. This, however, is not correct because it does not yet take the \emph{framing anomaly} into account. This framing anomaly also plays an important role for the definition of the moduli space of modular functors in \cite{brochierwoike}. In other words, without a proper treatment of the framing anomaly, one does not arrive at the correct notion of maps between modular functors.

\begin{remark}\label{[Origin of the anomaly for VOAs]}\label{remanomaly} To understand this subtlety, consider the following case. Let $\Cc=\Cc_V$ be the category of admissible modules for a $C_2$-cofinite and rational VOA.  Then, as mentioned in \cref{sec:VOAs-CB}, one can associate to $X_1, \dots, X_n \in \Cc_V$ a vector bundle of finite rank $\VV_g(X_1, \dots, X_n)$ on $\bTMgn$. When $g=0$, then this vector bundle is equipped with a flat connection (with logarithmic singularities along $\TDelta_{0,n}$). In fact, in this situation the line bundle $\Lambda$ is trivial and therefore the associated sequence \eqref{eq:atiyah} splits. However, as soon as $g \geq 1$ and the central charge $c$ of $V$ is non zero, then \eqref{eq:atiyah} is non split (for $\alpha=c/2$ and $\Lc=\Lambda$), and therefore the $\VV_g(X_1, \dots, X_n)$ admits only a \textit{projectively} flat connection (with logarithmic singularities along $\TDelta_{g,n}$). In other words, $\Tc_{g,n}(-\TDelta_{g,n})$ only acts on $\VV_g(X_1, \dots, X_n)$ projectively.
\end{remark} 

Because of the framing anomaly, it is not the mapping class group that acts on the spaces of conformal blocks, but an extension of the mapping class group. This can be included through the following definition (see \cite[Definition~3.5]{brochierwoike}, building on concepts in \cite{Segal,ms89,turaev,tillmann,baki}):

\begin{definition} \label{def:MF} 
A \emph{modular functor} is a pair $(\cat{Q},\Fk)$ of
\begin{itemize}

\item an extension $\cat{Q}$ of $\Surf$ in the sense of \cite[Section~3.1]{brochierwoike}, i.e.\ a modular operad $\cat{Q}$ with a map $\cat{Q}\to\Surf$ of modular operads with connected homotopy fiber, equipped with a section over genus zero, and satisfying an insertion of vacua property; 
\item a modular $\cat{Q}$-algebra $\Fk$ valued in $\Rexf$.
\end{itemize}
\end{definition}

\begin{remark}\label{remmfdef}
While this definition from \cite{brochierwoike} is informed by the classical definitions
(which already mutually disagree, see \cite{henriques} for this well-known problem), it is logically independent and formulates all axioms in a coherent way using the language of modular operads in symmetric monoidal biategories.
The theory of symmetric monoidal bicategories is developed rigorously in \cite{schommerpries}, and earlier definitions, most notably \cite{turaev,baki}, generally do not take such coherence isomorphisms into account, at least not comprehensively. Moreover, as far as the extensions of the mapping class groups are concerned (to the extent to which this point is actually even formalized), many older definition are directly tailored towards modular categories. The extensions in \cite{brochierwoike} satisfy topological requirements, but the definition is agnostic to the origin of the extensions. Finally, let us note that for the precise relation of \cite{turaev,baki} to \cite{brochierwoike}, it is not even clear how to formulate a comparison statement: Clearly, the definitions do not agree on the nose, so we could ask, at best, for a correspondence between modular functors \emph{up to equivalence of modular functors}. While in \cite{brochierwoike} there is actually a notion of equivalence of modular functors based on equivalences of modular algebras in bicategories~\cite[Section~2.4]{cyclic} and through localization at maps comparing different extensions \cite[Section~3.2]{brochierwoike}, a corresponding notion of equivalence of modular functors in~\cite{baki} that would allow for instance to relate modular functors over different extensions does not seem to be available.
\end{remark}

\section{The modular functor from coinvariants} \label{sec:CBareMF}

In this section we will show that sheaves of coinvariants associated to modules of a $C_2$-cofinite and rational VOA $V$ define a modular functor. As defined in \cref{def:MF}, we need an extension $\cat{Q}$ of $\Surf$ and a $\cat{Q}$-algebra valued in $\Rexf$.

\subsection{The extension} The extension of $\Surf$ that defines the modular functor in this context is part of a family of extensions $\Surf^\alpha$ depending on a scalar $\alpha \in \CC$. When the scalar is un-specified, this is denoted $\widetilde{\Surf}$ in \cite{deshpande.mukhopadhyay:2019} and $\widetilde{\Teich}$ in \cite{baki}. Instead of explaining the details of this extension, we describe what are the key features of a $\Surf^\alpha$-algebra valued in $\Rexf$. Using the perspective outlined in \cref{rmk:list}~\ref{rmk:Dmod}, if a $\Surf$-algebra gives rise to a family of compatible D-modules over $\bTMgn$ (with logarithmic singularities), a $\Surf^\alpha$-algebra gives rise to a family of \textit{twisted} D-modules over $\bTMgn$  (with logarithmic singularities). The \textit{twist} is specified in this case by the Hodge line bundle $\Lambda$ and the scalar $\alpha$. 

More explicitly, for every stable pair $(g,n)$, and for every $X_1, \dots, X_n \in \Cc$, a $\Surf^\alpha$-algebra $\Fk$ naturally assigns a vector bundle of finite rank $\Fk_{g}(X_1, \dots, X_n)$ on $\bTMgn$ equipped with an action of the Atiyah algebra $\alpha \Ac_\Lambda(-\TDelta_{g,n})$, i.e.\ an $\alpha\Ac_\Lambda(-\TDelta_{g,n})$-module of finite rank. We refer the reader to \cite[Section 7.1]{DGT1} and references therein for more details about Atiyah algebras.

The axioms that $\Fk$ is an algebra over $\Surf^\alpha$ imposes that, along with the collection of these bundles, one also needs to record  natural isomorphisms between these twisted D-modules as observed for instance in \eqref{eqnexcision} and \eqref{eq:pov}.

\subsection{The modular algebra}  Let $V$ be a $C_2$-cofinite and rational VOA and denote by $c$ its central charge. As in \cref{sec:VOAs-CB}, we also assume that $V$ and self-dual, i.e.\ that $V$ is \textit{strongly rational}. Let $\Cc_V$ be the category of admissible  $\NN$-graded $V$-modules. Assigning to a $V$-module $M$ its contragredient $M'$ gives rise to a non-degenerate pairing $\kappa$ on $\Cc_V$ which becomes symmetric through the canonical isomorphisms $M\cong M''$. The pair $(\Cc_V, \kappa)$ will be the category with non-degenerate symmetric pairing underlying our modular functor. 

We define the functor $\VV$ as the one associating, to every stable pair $(g,n)$, and every $X_1, \dots, X_n \in \Cc_V$, the sheaf of coinvariants $\VV_{g}(X_1, \dots, X_n)$ on $\bTMgn$.

\begin{remark} \label{rmk:g0} We note that the definition of $\VV_{g}(X_1, \dots, X_n)$ of \cite{FBZ:2004, DGT1,DGT2} does require that $n \geq 1$, but does not actually need that $(g,n)$ is stable. However, when $n=0$, one can follow the procedure described in \cite[Section 4.2, Equation (30)]{DGT1} and still obtain a well defined sheaf $\VV_g$. We will provide further detail about this in \cref{subsec:pov}.
\end{remark}

The main result of this paper is the following theorem.

\begin{theorem} \label{thm:VGisMF} Let $V$ be a VOA which is strongly rational (i.e.\ of CFT-type, $C_2$-cofinite, rational and self-contragredient). Then $(\Surf^{c/2},\VV)$ defines a modular functor.     
\end{theorem} 

As recalled in \cref{sec:CBrecall}, under the assumptions of \cref{thm:VGisMF}, the sheaves $\VV_{g}(X_1, \dots, X_n)$ are indeed vector bundles or finite rank over $\bTMgn$ which are naturally equipped with an action of $\frac{c}{2}\Ac_\Lambda(-\TDelta_{g,n})$ \cite{DGT1,DGT2}. We are left to show that these twisted D-modules satisfy appropriate compatibility conditions. We will do so in the next sections and wrap up the proof in \cref{sec:conclusion}.

\subsection{Propagation of vacua compatibility} \label{subsec:pov} We begin describing the analogue of \eqref{eq:pov}. We first of all note that the element $E$ appearing in \eqref{eq:pov} is given the vacuum representation of the VOA $V$, that is the VOA $V$ itself. Indeed, we have that $\VV_0(X)=\CC\delta_{X=V}$ and $V=V'$ by assumptions (see e.g. \cite{FBZ:2004,zhu:global}). We now rephrase the relation between $\Sigma'$ and $\Sigma''$ using $\bTMgn$ instead of $\Surf$ so that \eqref{eq:pov} can be described in terms of twisted D-modules.

As noted in \cref{rmk:Surf=TM}, we can translate the operation of capping a boundary component with a disk (in $\Surf$), with the map that forgets one of the marked points (in $\bTMgn$). Namely, denoting by
\[ \xi_{n+1} \colon \bTM_{g,n+1} \to \bTM_{g,n}
\] the map that forgets the datum of the $n+1$-st point, equation \eqref{eq:pov} is translated into an isomorphism
\begin{equation} \label{eq:povtoprove} \xi_{n+1}^* \VV_g(X_1, \dots, X_n) \cong \VV_g(X_1, \dots, X_n, V)
\end{equation} of $\frac{c}{2}\Ac_\Lambda(-\TDelta_{g,n})$-modules.

To show that \eqref{eq:povtoprove} holds, consider $(C,P_\bullet, \tau_\bullet)$ be a curve over $S$ (i.e.\ an element of $\bTMgn(S)$). For every point $Q$ disjoint from $P_\bullet$ and a non-zero tangent vector $\zeta$ at $Q$ we need to provide an isomorphism 
\[ \VV(X_1, \dots, X_n)_{[C,P_\bullet, \tau_\bullet]} \cong \VV(X_1, \dots, X_n, V)_{[C,P_\bullet\sqcup Q, \tau_\bullet \sqcup \zeta]}  
\] of twisted D-modules over $S$. Let $\mathbf{1}$ denote the vacuum element of $V$. This is a preferred element of degree zero of $V$ which plays a role similar to that of the identity in an associative algebra, and it is also annihilated by the Virasoro operators $L_i$ whenever $i \geq -1$. As recalled in \cite[Theorem 4.3.1]{DGT2} the map
\[ X_1 \otimes \dots \otimes X_n \to X_1 \otimes \dots \otimes X_n \otimes V,  \qquad \mathrm{x}_\bullet \mapsto \mathrm{x}_\bullet \otimes \mathbf{1}
\] induces an isomorphism of the associated sheaves of coinvariants
\[ \Phi \colon \VV(X_1, \dots, X_n)_{[C,P_\bullet, \tau_\bullet]} \overset{\cong}\longrightarrow \VV(X_1, \dots, X_n, V)_{[C,P_\bullet\sqcup Q, \tau_\bullet \sqcup \zeta]}.  
\] 
Summarizing, in order to show that \eqref{eq:povtoprove} holds, we are left to show the following statement. 

\begin{proposition} \label{prop:pov} The map $\Phi$ is compatible with the action of $\frac{c}{2}\Ac_\Lambda(-\TDelta_{g,n})$.
\end{proposition}

\begin{proof} To show that this holds, it is helpful to understand how the projective connection arises in slightly more explicit terms. For simplicity, we assume that $C$ is a family of smooth curves over $S$ (the extension to the nodal case is straightforward). Then there is an exact sequence:
\[ 0 \longrightarrow \Oc_S \longrightarrow \alpha\Ac_\Lambda \longrightarrow \Tc_S \longrightarrow 0
\] 
Furthermore, in view of the uniformization theorem, if the curve $C$ is marked by points $(P_1, \dots, P_n)$ with coordinates $(t_1, \dots, t_n)$ (lifting the tangents $\tau_i$) we also have the exact sequence of $\Oc_S$-modules
\[ 0 \longrightarrow \Tc_{C}(C \setminus P_\bullet) \longrightarrow \bigoplus_{i=1}^n \Oc_S(\!(t_i)\!)\partial_{t_i} \longrightarrow \Tc_S \longrightarrow 0.
\] Recall moreover that the Virasoro algebra $\Vir$ is an extension of $\CC(\!(t)\!)\partial_t$. By tensoring with $\Oc_S$, we get therefore, for every $i \in \{1, \dots, n\}$, the exact sequence
\[ 0 \longrightarrow \Oc_S \longrightarrow \Vir \hat{\otimes} \Oc_S \longrightarrow  \CC(\!(t_i)\!)\partial_{t_i} \longrightarrow 0.\]
One can show that the sequences above combine into the following commutative diagram
\[ \xymatrix{
& &  \Tc_{C}(C \setminus P_\bullet) \ar@{=}[r] \ar[d] & \Tc_{C}(C \setminus P_\bullet) \ar[d] \\ 
0 \ar[r]& \Oc_S \ar[r] \ar@{=}[d] & n\Vir_S \ar[r] \ar@{->>}[d] & \bigoplus_{i=1}^n \Oc_S(\!(t_i)\!)\partial_{t_i} \ar[r] \ar@{->>}[d] &  0\\
0 \ar[r]& \Oc_S \ar[r] & \alpha \Ac_\Lambda \ar[r] & \Tc_S \ar[r] &  0,
}\] 
where $n\Vir_S$ is the quotient of $\bigoplus_{i=1}^n \Vir \,\hat{\otimes}_\CC \Oc_S$ which identifies all the centers of each individual Virasoro algebra $\Vir$.

By definition, every $V$-module $X_i$ has an action of the Virasoro algebra with central charge $\alpha=c/2$ and therefore  $n\Vir_S$ naturally acts on $X_1 \otimes \dots \otimes X_n \otimes \Oc_S$. As shown in \cite{DGT1}, this action descends to an action of $\alpha\Ac_\Lambda$ on $\VV(X_1, \dots, X_n)_{[C,P_\bullet, \tau_\bullet]}$, defining in this way the projective connection. 

The action of $\alpha \Ac_\Lambda$ on $\VV(X_1, \dots, X_n, V)_{[C,P_\bullet\sqcup Q, \tau_\bullet \sqcup \zeta]}$ can be similarly described by considering $(n+1)\Vir_S$ instead of $n\Vir_S$ to take into account the extra point $Q$. However, instead of $(n+1)\Vir_S$, one can actually compute the connection via the Lie subalgebra $n\Vir_S \oplus \Vir_S^+$, where $\Vir_S^+\cong \Oc_S[\![z]\!] \partial_z$ for $z$ a local coordinate at $Q$ lifting $\zeta$. 

To finish the proof, it is enough to show that given an element $v \in n\Vir_S $ and $w \in \Vir^+_S$ and for every $\mathrm{x}_\bullet \in X_1 \otimes \dots \otimes X_n$, we have
\[ (v,w)*(\mathrm{x}_\bullet \otimes \mathbf{1}) = v*(\mathrm{x}_\bullet) \otimes \mathbf{1}.
\] 
Since 
\[(v,w)*(\mathrm{x}_\bullet \otimes \mathbf{1}) = v*(\mathrm{x}_\bullet) \otimes \mathbf{1} + \mathrm{x}_\bullet \otimes w(\mathbf{1}),
\] it remains to show that $\Vir^+_S$ acts trivially on $\mathbf{1}$. This is true because $\Vir^+_S$ is spanned, over $\Oc_S$, by operators $L_i$ with $i \geq -1$ and all of them act trivially on $\mathbf{1}$ by definition. \end{proof}

\subsection{Gluing compatibility} \label{sec:gluing} We now show that also \eqref{eqnexcision} holds true for $\Fk=\VV$. First of all note that since $\Cc_V$ is semisimple, we have that  \[\Delta = \bigoplus_{S \text{ simple}} S \boxtimes S',\]
where $S'$ is the contragredient module of $S$ and so we actually want to establish \eqref{eqnexcisionSS}. As in the previous section, we will do so in the context of twisted D-modules over moduli of curves. For this, we will need to introduce some notation. 
	
Let $\Gamma$ be a connected graph with only one edge $e$. Let $T=\nu(\Gamma)$ and $T'=\pi_0(\Gamma)$, so that $\Gamma$ can be interpreted as an element of $\Hom_{\Graphs}(T,T')$. As detailed in \cite{baki} and further expanded in \cite[Appendix A]{deshpande.mukhopadhyay:2019}, to such datum one associates a natural \emph{specialization map} 
\[ \Sp_{e} \colon \left(\frac{c}{2}\Ac_{\Lambda}\text{-mod}\right)(\bTM_{T'}) \to \left(\frac{c}{2}\Ac_{\Lambda}\text{-mod}\right)(\bTM_{T}).
\] In the expression above, for $T=T_n$ we set 
\[\bTM_{T_n} = \bigsqcup_{\substack{g \in \NN \text{ s.t.} \\(g,n) \text{ stable }}} \bTM_{g,n}.\]  When $T$ is a corolla, i.e.\ a disjoint union of $T_n$s, then $\bTM_T$ will be the product of the associated $\bTM_{T_n}$. Using this language, if $\Gamma$ has only one vertex then it follows that \eqref{eqnexcisionSS} is translated in requiring an isomorphism
\begin{equation} \label{eq:specomp} \Sp_e \VV_g(X_1, \dots, X_n) \cong \bigoplus_{S \text{ simple }} \VV_{g-1}(S', S, X_1, \dots, X_n).
\end{equation} which preserves the structures of $\frac{c}{2} \Ac_\Lambda$-modules. If instead $\Gamma$ has two vertices, and therefore $T=T^1 \sqcup T^2$ is disconnected, the required isomorphism will take the form
\begin{equation} \label{eq:specomp-g1g2} \Sp_e \VV_{g_1 + g_2}(X_1, \dots, X_{n_1}, Y_1, \dots, Y_{n_2}) \cong \bigoplus_{S \text{ simple }} \VV_{g_1}(S', X_1, \dots, X_{n_1}) \otimes \VV_{g_2}(S, Y_1, \dots, Y_{n_2}),
\end{equation} where $\{1, \dots, n_1\}$ and $\{1, \dots, {n_2}\}$ are labels of the legs of $T^1$ and $T^2$ respectively.

\begin{proposition} \label{prop:glueholds} There exist isomorphisms \eqref{eq:specomp} and \eqref{eq:specomp-g1g2} of $\frac{c}{2}\Ac_\Lambda$-modules.    
\end{proposition}

\begin{proof} We will only consider \eqref{eq:specomp} as \eqref{eq:specomp-g1g2} proceeds in an analogue way. To show the claim, we adapt the argument introduced in  \cite[Sections 7.6--7.8]{baki} and well detailed in \cite[Section 16 and Appendix A]{deshpande.mukhopadhyay:2019}. One starts with a curve $C$ over a smooth variety $S$ (together with marked points $P_\bullet$ and non-zero tangent directions $\tau_\bullet$) such that $C$ is nodal along a smooth divisor $D \subset S$ and such that $C|_{D}$ has only one node (this condition correspond to the fact that $\Gamma$ has only one edge).  In view of how $\Sp_e$ is defined, in order to show that \eqref{eq:specomp} holds, one must show that there is a connection-preserving isomorphism between
\[ \Sp_e \VV(X_\bullet)_C
\] and the sheaf of coinvariants 
\[  \bigoplus_{S \text{ simple }} \VV(X_\bullet, S, S')_{\widetilde{C}},
\] where $\widetilde{C}$ is an auxiliary curve obtained from $C$ by normalizing the fiber of $C$ along $D$. 

Adapting the arguments used to prove \cite[Theorem 7.8.8]{baki} (see e.g. \cite[Proposition 7.8.6 and 7.8.7]{baki} and \cite[Proposition 16.5]{deshpande.mukhopadhyay:2019}), the gluing compatibility naturally follows from \cref{lem:key}, where an isomorphism \eqref{eq:specomp} is constructed on an infinitesimal neighborhood $D^{(n)}$ of $D$. \end{proof}

In order to conclude the proof of \cref{prop:glueholds}, we need to state and prove \cref{lem:key}, for which we need to set up some notation. Let $D^{(n)}$ be an infinitesimal neighbourhood of $D$ in $S$ and denote by $\VV(X_\bullet)_C^{(n)}$ the induced $\Oc_{D^{(n)}}$-module. Note that it can be equivalently defined as the sheaf of coinvariants associated with the restriction of $C$ to $D^{(n)}$. Similarly, we will use the notation $\widetilde{C}_{D^{(n)}}$ for the family of smooth curves over $D^{(n)}$ which is obtained normalizing $C$ along $D$. Finally, let $q=0$ be a local equation for $D$ in $S$, so that $D^{(n)}=\Spec(\Oc_S/(q^{n+1}))=\Spec(\Oc_D[\![q]\!]/(q^{n+1}))$. 

We recall that the contragredient module $S'$ of $S$ is a $V$-module which has as underlying vector space 
\[ S'= \bigoplus_{d \in \NN} \Hom_{\CC}(S_d, \CC).
\]
In particular, the vector space $S \otimes S'$ contains the elements $\ee_{S,d}$ corresponding to the identity map of $S_d$, seen as an element of $S_d \otimes \Hom_{\CC}(S_d, \CC)$.
We define the series
\[\ee_S(q) \coloneqq \sum_{d \in \NN} \ee_{S,d} q^d\] which naturally lives in $(S\otimes S')[\![q]\!]$. Tensoring with the element $\ee_S(q)$ defines the map
\[\EE_S \colon  (X_1 \otimes \dots \otimes X_n) \widehat{\otimes} \Oc_D[\![q]\!] \to (X_1 \otimes \dots \otimes X_n \otimes S \otimes S')\widehat{\otimes} \Oc_D[\![q]\!], \] which corresponds to a component of the map (47) of \cite{DGT2}.
Denote the truncation of $\EE_S$ to $D^{(n)}$ by  \[\EE_S^{(n)} \colon (X_1 \otimes \dots \otimes X_n)\otimes \Oc_D[\![q]\!]/(q^{n+1}) \to (X_1 \otimes \dots \otimes X_n \otimes S \otimes S')\otimes \Oc_D[\![q]\!]/(q^{n+1}).\] By taking the sum over the finitely many isomorphisms classes of $V$-modules we obtain the map
\begin{equation*}\label{eq:EE} \EE=\bigoplus_S \EE_S \colon (X_1 \otimes \dots \otimes X_n)\widehat{\otimes} \Oc_D[\![q]\!] \to \bigoplus_{S} (X_1 \otimes \dots \otimes X_n \otimes S \otimes S')\widehat{\otimes} \Oc_D[\![q]\!],
\end{equation*} and its truncation $\EE^{(n)}$. The statement that will imply \cref{prop:glueholds} is the following.

\begin{lemma} \label{lem:key} The $\Oc_{D^{(n)}}$-linear map
\[\EE^{(n)} \colon X_\bullet\otimes \Oc_{D^{(n)}} \to \bigoplus_{S} X_\bullet \otimes S \otimes S' \otimes \Oc_{D^{(n)}}, \quad \mathrm{x}_\bullet \otimes 1 \mapsto \sum_S \mathrm{x}_\bullet \otimes \ee_S(q) 
\] induces an isomorphism of $\Oc_{D^{(n)}}$-modules
\[ \VV(X_\bullet)_{C}^{(n)} \longrightarrow \bigoplus_{S} \VV(X_\bullet, S, S')_{\widetilde{C}_{D^{(n)}}}.
\] Furthermore one has the compatibility
\begin{equation} \label{eq:qdq} (q\partial_q) \circ \EE_S^{(n)} = \EE_S^{(n)} \circ q \partial_q -w_S \cdot \EE_S^{(n)}
\end{equation} where $w_S \in \QQ$ is the conformal weight of $S$.
\end{lemma}
\begin{proof}
The fact that the map $\EE(q)^{(n)}$ induces a map between sheaves of coinvariants follows from \cite[Theorem 8.5.1]{DGT2}, where it is also shown that the induced map is actually an isomorphism (called $\Psi$ in \cite{DGT2}). We are thus only left to check that \eqref{eq:qdq} holds true. 
Let $\mathrm{x}_\bullet \in X_\bullet$. Then using the Leibniz rule we obtain:

\begin{align*}
(q\partial_q) \circ \EE_S^{(n)}(\mathrm{x}_\bullet) & = (q\partial_q)(\mathrm{x}_\bullet \otimes \ee_S)\\
&= (q\partial_q)(\mathrm{x}_\bullet) \otimes \ee_S + \mathrm{x}_\bullet \otimes q \partial_q(\ee_S)\\
&=  \EE_S^{(n)} \circ q \partial_q (\mathrm{x}_\bullet) + 
\sum_{d=0}^n \mathrm{x}_\bullet \otimes (q \partial_q)(e_d \otimes \epsilon_d) q^d + \sum_{d=0}^n d (\mathrm{x}_\bullet \otimes e_d \otimes \epsilon_d)  q^d. 
\end{align*}
In the last equality, we wrote the truncation of $\ee_S$ as $\sum_{d =0}^n e_d \otimes \epsilon_d \; q^d$, where $e_d \in S_d$ and $\epsilon_d \in S'_d$ are dual to each other. To conclude, we need to explain how $q \partial_q$ acts on $S \otimes S'$. As in \cite[Example 7.8.3]{baki} we can see that the tangent vector $q\partial_q$ lifts to a tangent vector of $\widetilde{C}_{D^{(n)}}$ whose expansion around $Q_\pm$ is given by  $\alpha_\pm z_\pm \partial_{z_\pm}$  for some $\alpha_\pm \in \CC$ such that $\alpha_+ + \alpha_- =1$. Therefore we have that
\begin{align*}
(q \partial_q)(e_d \otimes \epsilon_d) & = (\alpha_+ z_+ \partial_{z_+} e_d) \otimes \epsilon_d  + e_d \otimes (\alpha_- z_- \partial_{z_-} \epsilon_d)\\
&= - \alpha_+ L_0(e_d) \otimes \epsilon_d - \alpha_- (e_d \otimes L_0(\epsilon_d))\\
&= -\alpha_+ (d+w_S) (e_d \otimes \epsilon_d)- \alpha_- (d+ w_S)(e_d \otimes \epsilon_d)\\
&= -(d+w_S) (e_d \otimes \epsilon_d).
\end{align*}
Combining the two computations, we obtain
\begin{align*}
(q\partial_q) \circ \EE_S^{(n)}(\mathrm{x}_\bullet) & = 
\EE_S^{(n)} \circ q \partial_q (\mathrm{x}_\bullet) - \sum_{d=0}^n (d+w_S -d ) (\mathrm{x}_\bullet \otimes e_d \otimes \epsilon_d) q^d  \\
&= \EE_S^{(n)} \circ q \partial_q (\mathrm{x}_\bullet) + w_S \EE_S^{(n)}(\mathrm{x}_\bullet) \ , 
\end{align*} which concludes the argument.
\end{proof}

\subsection{Proof of \texorpdfstring{\cref{thm:VGisMF}}{Theorem \ref{thm:VGisMF}}} \label{sec:conclusion} We have now all the ingredients to show that \cref{thm:VGisMF} holds. First of all, using \cref{prop:pov} and \cref{rmk:g0} one shows that to every stable pair $(g,0)$, the sheaf $\VV_g$ is well defined and it is naturally equipped with an action of $c/2 \Ac_\Lambda(-\TDelta_{g,0})$. Furthermore \cref{prop:pov} shows that this system of twisted D-modules with logarithmic singularities comes equipped with natural maps corresponding to forgetting marked points.

Moreover, \cref{prop:glueholds} tells us that this system of twisted D-modules with logarithmic singularities is equipped with natural isomorphisms corresponding to the operation of edge contraction. 

Since every morphism between the spaces $\bTMgn$ (induced from those in $\Graphs$) is obtained composing forgetting marked points and contracting edges, we obtain that $\VV$ indeed defines a well defined modular algebra over the operad $\Surf^{c/2}$. 
The fact that $\Surf^{c/2}$ qualifies as one of the extensions allowed in \cref{def:MF} follows from \cref{remanomaly} and Proposition~\ref{prop:pov}.
\qed

\section{Consequences and further questions} \label{sec:final} We describe here some consequences of \cref{thm:VGisMF} as well as open questions and direction of further investigation. We begin with a background section which will be helpful to understand some of the consequences of \cref{thm:VGisMF}.

\subsection{Ribbon Grothendieck--Verdier duality}   \label{sec:GVbackground}
Surfaces of genus zero form a cyclic operad in the sense of \cite{gk} which is equivalent to the operad of framed little disks discussed e.g.\ in \cite{salvatorewahl}.
As a result, any modular functor yields, by genus zero restriction, a cyclic framed $E_2$-algebra, see \cite[Section~7.3]{cyclic} for the technical details. By \cite[Theorem 5.12]{cyclic} a $\Rexf$-valued cyclic framed $E_2$-algebra amounts exactly to a category $\Cc \in \Rexf$ with the following structure:
\begin{itemize}
\item a monoidal product $\otimes \colon  \Cc\boxtimes \Cc\to \Cc$ with unit $I$, such that the hom functors $\Hom_\Cc(-\otimes X,K)$ are via representable  $\Hom_\Cc(-\otimes X,K)\cong \Hom_\Cc(-,DX)$ such that $D\colon \Cc^\op \to \Cc$ is an anti-equivalence;
\item a braiding $c_{X,Y} \colon X \otimes Y \ra{\cong}Y\otimes X$ for $\otimes$;
\item a balancing, i.e.\ natural automorphisms $\theta_X \colon X\to X$ satisfying that $\theta_I=\id_I$ for the monoidal unit $I$ and $\theta_{X\otimes Y}=c_{Y,X}c_{X,Y} (\theta_X \otimes \theta_Y)$ for $X,Y\in \Cc$, that additionally satisfies $\theta_{DX}=D\theta_X$.
\end{itemize}
This data on $\Cc$ is what in \cite{bd} is called a \emph{ribbon Grothendieck--Verdier structure}.
Therefore, any modular functor produces by genus zero restriction a ribbon Grothendieck--Verdier category \cite[Theorem 7.17]{cyclic}. Thus \cref{thm:VGisMF} gives us:

\begin{corollary} \label{cor:tensor} The modular functor $\VV$ naturally induces on $\Cc_V$ a \textit{ribbon Grothendieck--Verdier} structure. \qed
\end{corollary}

This ribbon Grothendieck--Verdier structure is inherited from the topology. A priori, we do not know how it relates to the algebraically constructed ribbon Grothendieck--Verdier structure in \cite{alsw} based on \cite{HL:I,HL:II,HL:III,HL:IV}, see \cref{sec:HL} for more details.
In any case, let us observe that in \cref{cor:tensor}, as well as in~\cite{alsw}, the Grothendieck--Verdier duality takes the contragredient representation, see also \cref{remgvduality} below.

\label{sec:tensor} 
We will now describe explicitly the monoidal product featuring in \cref{cor:tensor} arising from $\VV$, and only briefly discuss how the braiding and balancing arise. Recall that the antiequivalence $D \colon \Cc_V^\op \to \Cc_V$ is given by sending a module $M$ to its contragredient $M'$. First of all the monoidal product induced by $\VV$, denoted
\begin{equation} \label{eq:tensor}  - \cbotimes - \colon \Cc_V \boxtimes \Cc_V \to \Cc_V,
\end{equation} is given by \[ \qquad M \boxtimes N \mapsto M \cbotimes N \coloneqq \bigoplus_{S \text{ simple}} \VV_0(M, N, S')\otimes_\CC S,
\] where $M \cbotimes N$ naturally inherits the structure of a $V$-module since it is a direct sum of $V$-modules. 

\begin{remark} One can see that this implies that there are natural isomorphisms $X \cbotimes V \cong X \cong X \cbotimes V$ for every $X \in \Cc_V$, making $V$ the monoidal unit of $(\Cc_V, \cbotimes)$. Since the category $\Cc_V$ is semisimple, it is enough to show this for simple elements $X$. By definition
\[ X \cbotimes V = \bigoplus_{S \text{ simple}} \VV_0(X, V, S')\otimes_\CC S\ , \]
and Propagation of Vacua (see \cref{subsec:pov}) provides an isomorphism
\[ \VV_0(X, V, S') \cong \VV_0(X,S').\]
Furthermore, in view of \cite[Proposition 7.2]{zhu:global}, one can naturally identify $\VV_0(X,S') \cong  \Hom_{\Cc_V}(X',S')$. Since both $X$ and $S$ are simple, we have $\Hom_{\Cc_V}(X',S') \cong \CC \delta_{S=X}$, concluding the argument.
\end{remark}

For the category $\Cc_V$ to be a monoidal category, one also needs to specify the datum of an \textit{associator}, i.e.\ an isomorphism
\[ \bigoplus_{X,Y  \text{ simple}} \VV_0(A,B,X')\otimes_\CC \VV_0(X,C, Y') \otimes_\CC Y \cong \bigoplus_{X,Y  \text{ simple}} \VV_0(A,Y',X')\otimes_\CC \VV_0(X, B,C) \otimes_\CC Y.
\] In this situation, the associator arises from the composition of isomorphisms
\begin{equation} \label{eq:AB}\bigoplus_{X \text{ simple}} \VV_0(A,B,X')\otimes_\CC \VV_0(X,C, Y') \cong  \VV_0(A,B,C,Y')
\end{equation}
and \begin{equation} \label{eq:BC} \VV_0(A,B,C,Y') \cong \bigoplus_{X \text{ simple}} \VV_0(A,Y',X')\otimes_\CC \VV_0(X, B,C)   
\end{equation}
 which are obtained from the isomorphisms given in \eqref{eqnexcisionSS} induced from the maps in $\Surf_0$ corresponding to the morphisms in $\Graphs$ given in Figure~\ref{figure:associator}.
 
\begin{figure}[h!]
\centering
\includegraphics[width=0.66\textwidth]{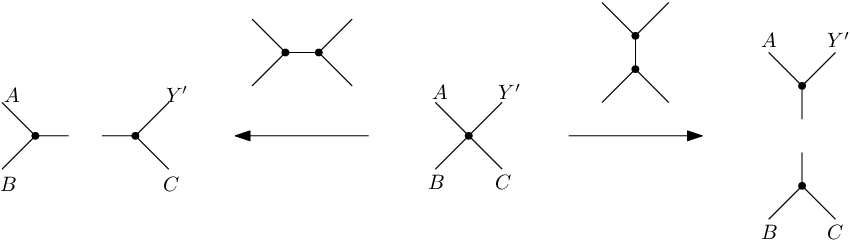}
\caption{Maps in $\Graphs$ inducing the associator in $\Cc_V$.}
\label{figure:associator}
\end{figure}

In other words, the fact that $\VV$ is a modular functor means that $\VV$ comes equipped with natural isomorphisms between the space of coinvariants associated to a marked curve $(C, P_\bullet, \tau_\bullet)$ and the space of coinvariants associated to all possible degenerations of $(C, P_\bullet, \tau_\bullet)$. Given a $\PP^1$ marked by $4$ distinct points (with non-zero tangent data), we can degenerate it to the union of two $\PP^1$ meeting at a node in more than one way (depending which points end up in which copy of $\PP^1$). If the points are labeled by $(A, B, C, Y')$, then splitting them into $(A,B)$ and $(C,Y')$ will give the isomorphism \eqref{eq:AB}, while splitting them into $(A,Y')$ and $(B,C)$ gives the isomorphism \eqref{eq:BC}.

Furthermore, since the fundamental group of $\bTM_{0,n}$ is the ribbon braid group on $n-1$ strands, we have that $\VV_0$ gives rise to a natural \textit{braiding} $X \cbotimes Y \cong Y \cbotimes X$ using $n=3$ and to the \textit{balancing} condition using $n=2$.

\begin{remark}
    The Grothendieck ring of $\Cc_V$ has already been studied in \cite{DGT3}, where it is shown that coinvariants from sufficiently nice VOAs (as those considered here and called \textit{of CohFT type} in \cite{DGT3}) can be used to define a Cohomological Field Theory via their Chern character. Under this perspective, we can interpret \cref{thm:modularfusion} and \cref{cor:tensor} as an enhancement of \cite{DGT3}. 
\end{remark}

\subsection{Relation to factorization homology}
\emph{Factorization homology}~\cite{higheralgebra,AF} allows one, roughly speaking, to integrate an algebra over the (framed) little $n$-disks operad over an $n$-dimensional oriented manifold. 
The concept takes inspiration from Beilinson--Drinfeld's notion of a \emph{chiral algebra}~\cite{bdca}, which is closely related to vertex operator algebras. In \cite{cg-fa} a relation between vertex (operator) algebras and factorization algebras---a concept closely related to factorization homology---is provided. The connection between factorization homology and spaces of conformal blocks in a topological sense is discussed \cite{brochierwoike}, where modular functors are classified using factorization homology.

The combination of \cite{brochierwoike} with \cref{thm:comparisonrt} gives us the following rather straightforward description of $\mathbb{V}$ via factorization homology:
For a compact oriented surface $\Sigma$, factorization homology of $\Cc_V$ gives us a linear category $\int_\Sigma \Cc_V$ together with a homotopy fixed point $\mathcal{O}_\Sigma^{\Cc_V} \in \int_\Sigma \Cc_V$, the so-called \emph{quantum structure sheaf}~\cite{bzbj}.
If $\Sigma$ is a disk, then $\int_\Sigma \Cc_V$ is simply $\Cc_V$, and the quantum structure sheaf just $V$.
In very intuitive terms, $\mathcal{O}_\Sigma^{\Cc_V} \in \int_\Sigma \Cc_V$ arises from ``globalizing'' the ``local'' object $V$ in the factorization homology of a disk.

Let $\Sigma$ be an oriented surface that, for simplicity, is closed, and let $H$ be a three-dimensional compact oriented handlebody with $\partial H=\Sigma$. By \cite[Section~4.1]{brochierwoike} the ansular functor~\cite{10.1093/imrn/rnad178} associated to $\Cc_V$ (the extension of the cyclic framed $E_2$-algebra $\Cc_V$ to handlebodies) induces a functor
\begin{align*}
    \Phi_{\Cc_V}(H) \colon \int_\Sigma \Cc_V \to \mathsf{vect} 
\end{align*}
from the factorization homology $\int_\Sigma \Cc_V$ to the category of finite-dimensional vector spaces.
We call $\Phi_{\Cc_V}(H)$ the \emph{generalized skein module} for $H$ and $\Cc_V$.
The following is a consequence of the classification of modular functors in \cite{brochierwoike} and the fact that $\Cc_V$ arises as the evaluation of a modular functor via restriction to genus zero:

\begin{corollary}\label{corconnected}
$\Cc_V$ is connected in the sense of \cite{brochierwoike}, i.e.\ $\Phi_{\Cc_V}(H)\cong \Phi_{\Cc_V}(H')$ for all handlebodies $H$ and $H'$ with boundary $\Sigma$. Moreover, the value of $\Phi_{\Cc_V}(H) \colon \int_\Sigma \Cc_V\to \mathsf{vect}$ on $\mathcal{O}_\Sigma^{\Cc_V}$ is canonically isomorphic to $\VV(\Sigma)$,
\begin{align*}\Phi_{\Cc_V}(H) \left(\mathcal{O}_\Sigma^{\Cc_V}\right) \cong \VV(\Sigma), \end{align*} as projective mapping class group representation, with the action on $\Phi_{\Cc_V}(H)  \left(\mathcal{O}_\Sigma^{\Cc_V}\right)$  induced by the homotopy fixed point structure on $\mathcal{O}_\Sigma^{\Cc_V}$ and connectedness. 
\end{corollary}

In summary, the generalized skein module translates the quantum structure sheaf into the space of conformal blocks and thereby encodes the projective mapping class group representation through factorization homology.

\subsection{\texorpdfstring{$\Cc_V$}{CV} is a modular fusion category}

By analyzing more carefully the properties of $\Cc_V$ we can strengthen the result of \cref{cor:tensor} and show  that actually $\Cc_V$ is a \emph{modular fusion category}. To this end, recall that  a modular fusion category is a finitely semisimple category $\cat{A}$ with rigid monoidal product, simple unit, braiding, balancing $\theta$ such that $\theta_{X^\vee}=\theta_X^\vee$ (here $-^\vee$ is the rigid duality),
simple unit and non-degenerate braiding.
Non-degeneracy of the braiding means that the subcategory $Z_2(\cat{A})\subset \cat{A}$ of all $X\in\cat{A}$ such that $c_{Y,X}\circ c_{X,Y}=\id_{X\otimes Y}$ 
is trivial, i.e.\ generated by the monoidal unit under finite direct sums. The subcategory $Z_2(\cat{A})$ is also called the
\emph{Müger center} of $\cat{A}$.

\begin{theorem} \label{thm:modularfusion} The modular functor $\VV$ naturally induces on $\Cc_V$  the structure of a modular fusion category. 
\end{theorem}

A result obtaining a modular tensor structure from a finitely semisimple modular functor whose ribbon Grothendieck--Verdier category is an $r$-category is 
given in \cite[Corollary~1.4]{etingof.penneys:2024:rigidity} by combining \cite[Theorem~5.5.1]{baki} with 
\cite[Corollary~1.3]{etingof.penneys:2024:rigidity}.
For the concerns mentioned in \cref{remmfdef},
the proof of \cite[Corollary~1.4]{etingof.penneys:2024:rigidity} does not apply in our framework, so we give a different proof in which we will, however, still rely on the purely algebraic statement \cite[Corollary~1.3]{etingof.penneys:2024:rigidity}.

\begin{proof}
The ribbon Grothendieck--Verdier category $\Cc_V$ from \cref{cor:tensor} is finitely semisimple. Moreover, it is an $r$-category in the sense of \cite[Definition 1.5]{bd} because $V$ is self-dual.
By \cite[Corollary~1.3]{etingof.penneys:2024:rigidity} this implies that the monoidal product of $\Cc_V$ is rigid.
A priori, the Grothendieck--Verdier duality $D$ could be different from the rigid duality, but by \cite[Proposition 2.3]{bd} $D$ agrees with the rigid duality up to tensoring with an invertible object. This invertible object must be the monoidal unit $I$ because $\Cc_V$ is an $r$-category.
This proves that the Grothendieck--Verdier duality of $\Cc_V$ agrees with the rigid duality.
Since $I$ is simple, this makes $\Cc_V$ a ribbon fusion category.

We are therefore left to show that the Müger center $Z_2(\Cc_V)$ of $\Cc_V$ is trivial which, since $\Cc_V$ is semisimple, is equivalent to showing that $Z_2(\Cc_V)$ has exactly one simple object. We first observe that
\[\# \, \text{simple objects in $Z_2( \Cc_V)$} = \dim  \text{HH}_0 (Z_2(   \Cc_V)),
\] 
where $\text{HH}_0 (Z_2(\Cc_V))$ denotes the zeroth Hochschild homology of $Z_2(\Cc_V)$. In view of \cite[Proposition 4.4 and Lemma~4.5]{skeinfin}, this dimension be computed as the dimension of the skein module $\mathsf{Sk}_{\Cc_V}(\SS^2\times \SS^1)$ for $\Cc_V$ on the three-dimensional manifold $\SS^2\times \SS^1$ (see e.g., \cite[Section~2.1]{skeinfin} for a modern introduction to skein modules in the spirit of the classical text~\cite{Walker}).

Let us finish the proof by showing that $\mathsf{Sk}_{\Cc_V}(\SS^2\times \SS^1)$ is one-dimensional:
For any compact oriented three-dimensional manifold $M$ with Heegaard splitting $M=H'\cup_\Sigma H$,
\cite[Corollary~3.10]{mwskein} tells us
\begin{align*}
\mathsf{Sk}_{\Cc_V} (M)\cong \int^{X \in \int_{\Sigma}      \Cc_V   } \Phi_{\Cc_V}(H')X^\vee \otimes \Phi_{\Cc_V}(H)X \ . 
\end{align*} With \cref{corconnected}, it is now an immediate consequence that $\mathsf{Sk}_{\Cc_V} (M)\cong \mathsf{Sk}_{\Cc_V} (M')$ if $M$ and $M'$ admit a Heegaard splitting for the same surface (simply because the $\Phi_{Cc_V}(H)$ do not depend on $H$).
In particular, for $\Sigma=\mathbb{T}^2$, we find $\mathsf{Sk}_{\Cc_V}(\SS^2\times \SS^1)\cong\mathsf{Sk}_{\Cc_V}(\SS^3)$.
Since the quantum trace gives us an isomorphism $\mathsf{Sk}_{\Cc_V}(\SS^3)\cong k$, it follows that $\mathsf{Sk}_{\Cc_V}(\SS^2\times \SS^1)$ is one-dimensional, concluding the argument.
\end{proof}

\begin{remark}\label{remgvduality}
Since the braiding of $\Cc_V$ is non-degenerate,
the cyclic structure on $\Cc_V$ (here given by taking the contragredient representation) is actually unique relative to the balanced braided monoidal structure~\cite[Corollary 4.5]{mwcenter}. 
\end{remark}

\subsection{Comparison with Huang--Lepowsky's tensor product\label{sec:HL}}  
Understanding the properties of categories of $V$-modules is a central question in the theory of VOAs. In particular, in \cite{HL:I,HL:II,HL:III,HL:IV} a monoidal structure on $\Cc_V$, which we denote $\hlotimes$ here, is defined using analytic methods and intertwiners. There is a lot of evidence that indeed the monoidal categories $(\Cc_V, \cbotimes)$ and $(\Cc_V, \hlotimes)$ are naturally isomorphic. One can in fact deduce that, under our assumptions on $V$ and $\Cc_V$ one has
\[ M \hlotimes N = \bigoplus_{S \text{ simple}} \cat{V}^S_{M,N} \otimes_\CC  S,
\] where $\cat{V}^S_{M,N}$ is the space of \textit{intertwining operators} of type $\left( \substack{S\\ M \, N} \right)$. Furthermore \cite[Proposition 7.4]{zhu:global} provides a natural (functorial) isomorphism 
\begin{equation} \label{eq:isocat} \cat{V}^S_{M,N} \overset{\cong}{\longrightarrow} \Hom_\CC(\VV_0(M,N,S'),\CC),
\end{equation} which gives therefore a natural isomorphism $M \hlotimes N \cong M \cbotimes N$ 
that preserves the monoidal unit $V$.  Therefore, in order to compare $(\Cc_V, \cbotimes)$ and $(\Cc_V, \hlotimes)$ as monoidal categories, it remains to show compatibility with the associators. A reason supporting that this is indeed the case is that, as we explain in \cref{sec:tensor} for $\cbotimes$ and as described in \cite{HL:IV} for $\hlotimes$, both associators are naturally constructed from the decomposition of a sphere with four punctures into spheres with three punctures in two different ways. We therefore expect the following result, which we plan to explore in future work.

\begin{conj} The map \eqref{eq:isocat} induces an equivalence of monoidal categories between $(\Cc_V, \cbotimes)$ and $(\Cc_V, \hlotimes)$.
\end{conj}

One strategy to show this would be to directly compare the spaces of coinvariants on a sphere with four marked points with the three-fold tensor product defined by Huang and Lepowsky. We note that, in an analytic setting, this problem has been recently analyzed in \cite{moriwaki:2024:vertexoperatoralgebraparenthesized}.

\subsection{The 2dCFT/3dTFT correspondence and correlators}
\label{sec:CFT} 
The topological approach to conformal field theory follows a two-step procedure: 
\begin{enumerate}
    \item One organizes the spaces of conformal blocks of the conformal field theory into a modular functor.
    \item One evaluates the modular functor on the disjoint union $\bar{\Sigma} \sqcup \Sigma$ of surfaces $\Sigma$ together with the orientation-reversed surface $\bar \Sigma$ and tries to find all consistent system of correlators, i.e.\ conformal blocks that are invariant under the action of the mapping class group and gluing. Finding the correlators is a mathematically precise incarnation of \textit{solving} the conformal field theory.
\end{enumerate}  
For the Reshetikhin--Turaev type modular functor $\Fk_\Cc$ of a modular fusion category $\Cc$, the problem of classifying and constructing the consistent systems of correlators was achieved in the series of papers of Fuchs--Runkel--Schweigert, partially with Fjelstad~\cite{frs1,frs2,frs3,frs4,ffrs,ffrsunique} in terms of \textit{special symmetric Frobenius algebras} inside the modular fusion category in question.  This  is often referred to as the \emph{FRS Theorem}.
Let us recall that a \emph{special symmetric Frobenius algebra} $F\in \Cc$ inside a pivotal finite tensor category with monoidal unit $I$
is an associative unital algebra with a symmetric non-degenerate pairing, see e.g.\ \cite{fuchsstigner}. In particular, this entails that in addition to the product $\mu \colon F \otimes F \to F$ with unit $\eta \colon I\to F$, there is a coproduct $\delta \colon F \to F\otimes F$ with counit $\varepsilon \colon F \to I$. A symmetric Frobenius algebra $F$ is called \emph{special} if  $\mu\circ\delta=\id_F$ and $\varepsilon\circ \eta \neq 0$.
 This solution procedure of rational two-dimensional conformal field theory via three-dimensional topological field theory can be interpreted as a \emph{holographic principle}~\cite{kapustinsaulina}.

There is however an open question that is explained in \cite[Section~3.3]{csrcft}: The relation between the mapping class group representation of the Reshetikhin--Turaev type modular functor associated to a modular fusion category and the flat connection on the vector bundles of conformal blocks over $\TMgn$ is not known for every $g$. An instance where such a comparison was established is for the affine VOA $L_\ell(\mathfrak{sl}_N)$ in \cite{andersenuenoinv}. 

Via the results of this article, we can prove a more general comparison. Throughout, $\Cc_V$ will denote the modular fusion category arising from the modular functor $\VV$.

\begin{theorem}\label{thm:comparisonrt}
Up to a contractible space of choices, the modular functor $\VV$ is the unique modular functor extending the modular fusion category $\Cc_V$ from genus zero to all surfaces. In particular, there is a canonical equivalence
\begin{align*}
    \VV \xrightarrow{\ \simeq \ } \Fk_{\Cc_V}
\end{align*}
of modular functors between $\VV$ and the Reshetikhin--Turaev type modular functor $\Fk_{\Cc_V}$ of the modular fusion category $\Cc_V$. This affords an extension of $\VV$ to a once-extended three-dimensional topological field theory.
\end{theorem}

\begin{proof}
    Since $\Cc_V$ is a modular fusion category by \cref{thm:modularfusion} the statement follows from the classification of modular functors~\cite[Theorem~6.8 and Section~8.1]{brochierwoike}.
\end{proof}

We note that this establishes a universal topological property of the modular fusion category $\Cc_V$ arising from conformal blocks.

Admittedly, it might be preferable to formulate such a comparison not for $\Cc_V$, but instead for the Huang--Lepowsky tensor structure (defined on the same underlying category, but a priori differently). As explained in \cref{sec:HL}, such a comparison is currently out of reach. However, \cref{thm:comparisonrt} is enough to conclude that $\VV$ forms a modular functor that can be described as the Reshetikhin--Turaev type modular functor of \emph{some} modular fusion category.
For many practical purposes, this is entirely enough as the following applications will demonstrate.

The space of conformal blocks for the torus $\mathbb{T}^2$ (which in view of \cref{thm:comparisonrt} is identified with the vector spaces freely generated by the set of simple objects of $\Cc_V$) comes with a multiplication induced by the monoidal product of $\Cc_V$ (the one from \cref{cor:tensor}) or, equivalently, by the tensor product of Huang--Lepowsky. In fact these agree because, for the algebra structure to be comparable, the associators are not relevant.  In other words, $(\Cc_V, \cbotimes)$ and $(\Cc_V, \hlotimes)$ have the same \emph{Verlinde algebra}~\cite{verlinde}.  
For this Verlinde algebra, we can give an alternative proof of \cite[Theorem 5.2]{huangverlinde} which heavily relies on \cref{thm:comparisonrt}:

\begin{corollary} \label{cor:verlinde}
The automorphism of the space of conformal blocks for the torus associated to the mapping class
\begin{align*}
    S=\begin{pmatrix} 0&-1\\ 1 & \phantom{-}0\end{pmatrix}\in \mathsf{SL}(2,\ZZ)\cong \Map(\mathbb{T}^2)
\end{align*}
transforms the multiplication of the Verlinde algebra into a diagonal multiplication, i.e.\ the fusion multiplication lies in the mapping class group orbit of a diagonal multiplication.
\end{corollary}

\begin{proof}
The \emph{categorical version} of the Verlinde formula \cite{mooreseiberg,turaev} for a modular fusion category gives us such a statement for the $\mathsf{SL}(2,\mathbb{Z})$-representation built via the Reshetikhin--Turaev construction. 
\Cref{thm:comparisonrt} tells us that the mapping class group representations on the spaces $\VV$ arise actually by applying the Reshetikhin--Turaev construction to \emph{some} modular fusion category whose Verlinde algebra agrees with the one in \cite{huangverlinde}. Hence, the statement follows.
\end{proof}

Another consequence of \cref{thm:comparisonrt} is that the results of the recent preprint~\cite{godfard2025semisimplicityconformalblocks} 
can be used to see that the mapping class group representations produced by $\VV$
 are semisimple. 

\medskip

Our main application, however, is that \cref{thm:comparisonrt} ensures that the FRS Theorem applies because the latter is fortunately formulated for \emph{any} modular fusion category regardless of whether such a modular fusion category arises through the theory of Huang--Lepowsky. 
We will formulate this in \cref{thm:correlators} below, for which we will need to introduce a little notation. Given $C=(C,P_\bullet, \tau_\bullet) \in \bTMgn$, we will use the notation $\bar{C}=(\bar{C}, \bar{P}_\bullet, \bar{\tau}_\bullet)$ to denote the curve obtained from $C$ by reversing the direction of $\tau$, that is  $\bar{C}=(C, P_\bullet, -\tau_\bullet)$. Note that this operation induces a $\ZZ/2\ZZ$ action on $\bTMgn$ which, in $\Surf$, amounts to the operation sending an oriented surface $\Sigma$ to the surface $\bar{\Sigma}$ with the opposite orientation. Let $B \in \Cc_V \boxtimes \Cc_V$ and set
\[ \VV(B^n)_{[C,P_\bullet,\tau_\bullet]} \coloneqq \VV(B, \dots, B)_{[C \sqcup \bar{C};P_1, \bar{P}_1, \dots, P_n, \bar{P}_n; \tau_1, \bar{\tau}_1, \dots, \tau_{n}, \bar{\tau}_n]},  
\] so that each copy of $B$ is inserted at the pair of points $(P_i, \bar{P}_i)$. Note that when $(C,P_\bullet, \tau_\bullet)$ vary in $\TMgn$, the vector space $\VV(B^n)_{[C,P_\bullet,\tau_\bullet]}$ defines a vector bundle
$\VV_g(B^n)$ on $\TMgn$ which is naturally equipped with a flat connection (where the anomaly is trivial). Using this notation, \cref{thm:comparisonrt} implies, when combined with the FRS Theorem, the following:

\begin{corollary}\label{thm:correlators}
    Let $F$ be a special symmetric Frobenius algebra $F$ in $\Cc_V$ and let $\mathsf{B}(F)$ be the \textit{bulk algebra} $\mathsf{B}(F) \in \Cc_V\boxtimes \Cc_V$  associated to $F$. Then, for every $g$ and $n$,  $F$ gives rise to a system of flat sections $s_{g,n}^F$ of $\VV_g(\mathsf{B}(F)^n)$ which is compatible with gluing. 
\end{corollary}

\begin{remark} We collect some observations.\begin{enumerate}[label=(\alph*)] \item When $F=V$ is the monoidal unit, the underlying object of its bulk algebra $\mathsf{B}(F)$ can be identified with $\Delta = \oplus S \boxtimes S'$. Applying \cref{thm:correlators} to $F=V$ allows us to conclude the very non-trivial statement that non-zero flat sections for $\VV_g(\Delta^n)$ always exist. In the language of conformal field theory, this corresponds to the so-called \emph{Cardy case}, see \cite{FUCHS2018287} for more background.

\item The spaces of conformal blocks $\VV$ evaluated on $C\sqcup \bar C$ belong to a modular functor that by restriction to genus zero gives us  the modular fusion category $\Cc_V\boxtimes \bar{\Cc}_V$, where $\bar{\Cc}_V$ has the same monoidal product, but the inverse braiding and balancing; this is again an easy consequence \cite[Theorem~6.8]{brochierwoike}. 
Actually, the modular functor associated to $\bar{\Cc}\boxtimes\Cc$ for any modular fusion category is an open-closed modular functor (this means that one has also spaces of conformal blocks for surfaces with embedded intervals in their boundary), and all consistent systems of correlators compatible in a certain sense with this open-closed extension arise from special symmetric Frobenius algebras~\cite{ffrsunique}.
 \end{enumerate}
\end{remark}

\subsection{Beyond semisimplicity} \label{sec:beyond}
The previously cited works of Huang and Lepowsky have been extended, together with Zhang, beyond our assumptions to include, for instance, non-rational VOAs \cite{HLZ}. As established in \cite{alsw}, under the assumptions of Huang, Lepowsky and Zhang, the category of $V$-modules is a ribbon Grothendieck--Verdier category. Similarly, the definition of coinvariants $\VV$ given in \cite{FBZ:2004,NT,DGT1,DGT2}, which we used in \cref{sec:tensor} to define the tensor product $\cbotimes$ does not require that the VOA it depends on is rational (or even $C_2$-cofinite). In \cite{DGK1} it is shown that if $V$ is of CFT type and $C_2$-cofinite, then the sheaves of coinvariants $\VV(X_\bullet)$ are coherent on $\bTMgn$, and since they admit a connection on $\TMgn$, they give rise to vector bundles on $\TMgn$. We note, however, that, without the rationality assumption, it is not known whether the Sewing Theorem, crucially used in the proof of \cref{lem:key} to construct the gluing isomorphism of \cref{prop:glueholds}, holds in this setting. In \cite{DGK2}, a sufficient condition to ensure that a Sewing-type Theorem holds (called \textit{smoothing property}) is provided in terms of algebraic properties of the \textit{mode transition algebra} $\MTA=\MTA(V)$ of $V$, an associative algebra naturally associated with $V$.

As described in \cref{sec:GVbackground}, in order to obtain a ribbon Grothendieck--Verdier structure on $\Cc_V$ from conformal blocks, it would suffice to analyze the properties of sheaves of conformal blocks on the moduli spaces $\bTM_{0,n}$ of rational curves, and not consider curves of higher genus when showing an analogue of \cref{thm:modularfusion}. Using the language of \cref{sec:GVbackground}, it makes sense to pose the following question:

\begin{question} \label{qs:CBareE2algebras} Under which assumptions on $V$, do the associated conformal blocks $\VV$ define a cyclic framed $E_2$-algebra?
\end{question}

In the following example we discuss one possible starting point to tackle the above question.
\begin{exx}
Let $V=\pi_r$ be the rank-$r$ Heisenberg VOA (for $r \geq 1$). Since sheaves of coinvariants $\VV_0(X_\bullet)$ arising from $\pi_r$-modules $X_\bullet$ are coherent on $\bTM_{0,n}$ (see \cite{DG}), one can show that these are indeed vector bundles of finite rank over $\TM_{0,n}$ which are naturally equipped with a flat connection. Furthermore, as a consequence of the \textit{smoothing property} \cite{DGK2} (which holds for $\pi_r$ as explicitly computed in \cite[Section 7]{DGK2} and \cite[Section 6.1]{DGK3}), these sheaves extend to vector bundles to the whole space $\bTM_{0,n}$ (and the connection naturally extends to the boundary acquiring logarithmic singularities along $\TDelta_{0,n}$). One could further adapt the methods used in \cref{subsec:pov} to show that the analogue of \eqref{eq:pov} holds true. Moreover, the smoothing property mentioned above ensures that these bundles satisfy an analogue of \eqref{eq:specomp-g1g2}. Namely, there is an isomorphism
\begin{equation} \label{eq:specompMTA} \Sp_e \VV_0(X_1, \dots, X_{n_1}, Y_1, \dots, Y_{n_2}) \cong  \VV_{0}(\MTA, X_1, \dots, X_{n_1}, Y_1, \dots, Y_{n_2}),
\end{equation}
of \textit{vector bundles} over $\bTM_{0,n_1+1} \times  \bTM_{0,n_2+1}$, where the mode transition algebra $\MTA$ can be naturally viewed as an element of $\Cc_{\pi_r} \boxtimes \Cc_{\pi_r}$. We note that, in order to be able to relate \eqref{eq:specompMTA} with \eqref{eqnexcision}---and obtain in this case a cyclic framed $E_2$-algebra---one would need to show that indeed there is an isomorphism between the coevaluation element $\Delta$ and the mode transition algebra $\MTA$, and further that \eqref{eq:specompMTA} is an isomorphism of D-modules (and not merely of vector bundles).  \end{exx}

Let now $V$ be a more general VOA. The {smoothing property} introduced in \cite{DGK2} is what implies, under finiteness of conformal blocks, the existence of an isomorphism of vector bundles of the type \eqref{eq:specompMTA} (for every genus). Since this could, in principle, hold also beyond the case of Heisenberg VOAs, in order to answer \cref{qs:CBareE2algebras} it is natural to pose the following question: 

\begin{question} \label{qs:MTAcoend} What is the relation between $\MTA$ and $\Delta$? Under which conditions are they isomorphic?
\end{question}

When $V$ is rational, then it is true that $\MTA =\Delta$. In fact, as shown in \cite[Remark 3.4.6]{DGK2}, one naturally has that \[\MTA \cong \bigoplus_{S \text{ simple}} S \boxtimes S',\] and we saw that the same decomposition holds true for $\Delta$. We believe that \cref{qs:MTAcoend} has a positive answer in the case $V=\pi_S$. However, we know from \cite[Proposition 9.1.4]{DGK2} that the smoothing property fails for the $C_2$-cofinite, but not rational VOA $W(p)$. Although this does not imply that an isomorphism similar to \eqref{eq:specompMTA} cannot exist, it might be an evidence that the relation between $\Delta$ and $\MTA$ could be quite subtle.

It is also natural to understand to what extent $\VV$ defines a modular functor in every genus. However, to formulate a perhaps more sensible question, we will need to recall some results and terminology. We note that in the situation in which we have a not necessarily semisimple modular category (but which is still rigid), an associated modular functor can still be built through the construction of Lyubashenko
\cite{lyu,lyubacmp,lyulex} (see also \cite[Section~8.3]{brochierwoike} for a factorization homology description of this construction). However, if we are just given a ribbon Grothendieck--Verdier category that is not necessarily modular, we can generally only build an \emph{ansular functor}~\cite{10.1093/imrn/rnad178}, a version of the notion of a modular functor in which all surfaces are replaced with handlebodies. In some cases, these extend to modular functors~\cite[Example 8.8]{brochierwoike} even though the category is not modular;
in other cases, they do not~\cite[Example 11.10]{woike2024cyclicmodularmicrocosmprinciple}. 
In any case, one may obtain from a ribbon Grothendieck--Verdier category a pivotal Grothendieck--Verdier category by forgetting the braiding and from a such a datum, one may construct an \textit{open modular functor}~\cite{envas}. In view of this, we therefore also pose the following:

\begin{question}\label{qs:open} Under which assumptions on $V$, 
can one adapt the construction of \cite{FBZ:2004} to build an ``\textit{open version}'' of $\VV$? When would this define an open modular functor? 
\end{question}

We note that the construction of conformal blocks from \cite{FBZ:2004} naturally belongs to the realm of algebraic geometry, while the notion of open modular functor is purely topological. Therefore, one key step towards answering \cref{qs:open} is to understand how to phrase the notion of an open modular functor in the algebraic framework.

\vfill

\footnotesize

\noindent{\textsc{C. Damiolini, Department of Mathematics, University of Texas at Austin, Austin TX, USA}}\\
\href{mailto:chiara.damiolini@austin.utexas.edu}{\texttt{chiara.damiolini@austin.utexas.edu}}

\medskip

\noindent{\textsc{L. Woike, Université Bourgogne Europe, CNRS, IMB UMR 5584, F-21000 Dijon, France}}\\
\href{mailto:lukas.woike@u-bourgogne.fr}{\texttt{lukas.woike@ube.fr}}


\begin{thebibliography}{10}
	
	\bibitem{alsw}
	R.~Allen, S.~Lentner, C.~Schweigert, and S.~Wood.
	\newblock {Duality structures for module categories of vertex operator algebras
		and the {F}eigin {F}uchs boson}.
	\newblock {\em Selecta Math. New Ser.}, 31(36), 2025.
	
	\bibitem{andersenuenoinv}
	J.~E. Andersen and K.~Ueno.
	\newblock Construction of the {W}itten-{R}eshetikhin-{T}uraev {TQFT} from
	conformal field theory.
	\newblock {\em Invent. Math.}, 201(2):519--559, 2015.
	
	\bibitem{AF}
	D.~Ayala and J.~Francis.
	\newblock Factorization homology of topological manifolds.
	\newblock {\em J. Topol.}, 8(4):1045--1084, 2015.
	
	\bibitem{baki}
	B.~Bakalov and A.~Kirillov, Jr.
	\newblock {\em Lectures on tensor categories and modular functors}, volume~21
	of {\em University Lecture Series}.
	\newblock American Mathematical Society, Providence, RI, 2001.
	
	\bibitem{bdca}
	A.~Beilinson and V.~Drinfeld.
	\newblock {\em Chiral algebras}, volume~51 of {\em American Mathematical
		Society Colloquium Publications}.
	\newblock American Mathematical Society, Providence, RI, 2004.
	
	\bibitem{BFM}
	A.~Beilinson, B.~Feigin, and B.~Mazur.
	\newblock Notes on conformal field theory. 1991.
	\newblock Available at
	\href{https://www.math.stonybrook.edu/~kirillov/manuscripts/bfmn.pdf}{https://www.math.stonybrook.edu/~kirillov/manuscripts/bfmn.pdf}.
	
	\bibitem{bzbj}
	D.~Ben-Zvi, A.~Brochier, and D.~Jordan.
	\newblock Integrating quantum groups over surfaces.
	\newblock {\em J. Topol.}, 11(4):874--917, 2018.
	
	\bibitem{bd}
	M.~Boyarchenko and V.~Drinfeld.
	\newblock A duality formalism in the spirit of {G}rothendieck and {V}erdier.
	\newblock {\em Quantum Topol.}, 4(4):447--489, 2013.
	
	\bibitem{brochierwoike}
	A.~Brochier and L.~Woike.
	\newblock A classification of modular functors via factorization homology,
	2023.
	\newblock \href{https://arxiv.org/pdf/2212.11259}{arXiv:2212.11259}.
	
	\bibitem{codogni:POV}
	G.~Codogni.
	\newblock Vertex algebras and {T}eichm\"{u}ller modular forms, 2020.
	\newblock \href{https://arxiv.org/pdf/1901.03079}{arXiv:1901.03079}.
	
	\bibitem{costello}
	K.~Costello.
	\newblock The {A}-infinity operad and the moduli space of curves, 2004.
	\newblock \href{https://arxiv.org/pdf/math/0402015}{arXiv:0402015}.
	
	\bibitem{cg-fa}
	K.~Costello and O.~Gwilliam.
	\newblock {\em Factorization Algebras in Quantum Field Theory}, volume~31 of
	{\em New Mathematical Monographs}.
	\newblock Cambridge University Press, 2016.
	
	\bibitem{DG}
	C.~Damiolini and A.~Gibney.
	\newblock On global generation of vector bundles on the moduli space of curves
	from representations of vertex operator algebras.
	\newblock {\em Algebr. Geom.}, 10(3):298--326, 2023.
	
	\bibitem{DGK2}
	C.~Damiolini, A.~Gibney, and D.~Krashen.
	\newblock Conformal blocks on smoothings via mode transition algebras.
	\newblock {\em Comm. Math. Phys.}, 406:131, 2025.
	
	\bibitem{DGK1}
	C.~Damiolini, A.~Gibney, and D.~Krashen.
	\newblock Factorization presentations.
	\newblock In {\em Higher dimensional algebraic geometry---a volume in honor of
		{V}. {V}. {S}hokurov}, volume 489 of {\em London Math. Soc. Lecture Note
		Ser.}, pages 163--191. Cambridge Univ. Press, Cambridge, 2025.
	
	\bibitem{DGK3}
	C.~Damiolini, A.~Gibney, and D.~Krashen.
	\newblock Morita equivalences for {Z}hu's algebra, 2025.
	\newblock \href{https://arxiv.org/pdf/2403.11855}{arXiv:2403.11855}.
	
	\bibitem{DGT1}
	C.~Damiolini, A.~Gibney, and N.~Tarasca.
	\newblock Conformal blocks from vertex algebras and their connections on
	{$\overline{\mathcal M}_{g, n}$}.
	\newblock {\em Geom. Topol.}, 25(5):2235--2286, 2021.
	
	\bibitem{DGT3}
	C.~Damiolini, A.~Gibney, and N.~Tarasca.
	\newblock Vertex algebras of {C}oh{FT}-type.
	\newblock In {\em Facets of algebraic geometry. {V}ol. {I}}, volume 472 of {\em
		London Math. Soc. Lecture Note Ser.}, pages 164--189. Cambridge Univ. Press,
	Cambridge, 2022.
	
	\bibitem{DGT2}
	C.~Damiolini, A.~Gibney, and N.~Tarasca.
	\newblock On factorization and vector bundles of conformal blocks from vertex
	algebras.
	\newblock {\em Ann. Sci. \'Ec. Norm. Sup\'er. (4)}, 57(1):241--292, 2024.
	
	\bibitem{deshpande.mukhopadhyay:2019}
	T.~Deshpande and S.~Mukhopadhyay.
	\newblock Crossed modular categories and the {V}erlinde formula for twisted
	conformal blocks.
	\newblock {\em Camb. J. Math.}, 11(1):159--297, 2023.
	
	\bibitem{etingofostrik}
	P.~Etingof and V.~Ostrik.
	\newblock Finite tensor categories.
	\newblock {\em Mosc. Math. J.}, 4(3):627--654, 782--783, 2004.
	
	\bibitem{etingof.penneys:2024:rigidity}
	P.~Etingof and D.~Penneys.
	\newblock Rigidity of non-negligible objects of moderate growth in braided
	categories, 2024.
	\newblock \href{https://arxiv.org/pdf/2412.17681}{arXiv:2412.17681}.
	
	\bibitem{farbmargalit}
	B.~Farb and D.~Margalit.
	\newblock {\em A primer on mapping class groups}, volume~49 of {\em Princeton
		Mathematical Series}.
	\newblock Princeton University Press, Princeton, NJ, 2012.
	
	\bibitem{ffrs}
	J.~Fjelstad, J.~Fuchs, I.~Runkel, and C.~Schweigert.
	\newblock T{FT} construction of {RCFT} correlators. {V}. {P}roof of modular
	invariance and factorisation.
	\newblock {\em Theory Appl. Categ.}, 16:No. 16, 342--433, 2006.
	
	\bibitem{ffrsunique}
	J.~Fjelstad, J.~Fuchs, I.~Runkel, and C.~Schweigert.
	\newblock Uniqueness of open/closed rational {CFT} with given algebra of open
	states.
	\newblock {\em Adv. Theor. Math. Phys.}, 12(6):1283--1375, 2008.
	
	\bibitem{FBZ:2004}
	E.~Frenkel and D.~Ben-Zvi.
	\newblock {\em Vertex algebras and algebraic curves}, volume~88 of {\em
		Mathematical Surveys and Monographs}.
	\newblock American Mathematical Society, Providence, RI, second edition, 2004.
	
	\bibitem{FHL}
	I.~B. Frenkel, Y.-Z. Huang, and J.~Lepowsky.
	\newblock On axiomatic approaches to vertex operator algebras and modules.
	\newblock {\em Mem. Amer. Math. Soc.}, 104(494):viii+64, 1993.
	
	\bibitem{FUCHS2018287}
	J.~Fuchs, T.~Gannon, G.~Schaumann, and C.~Schweigert.
	\newblock The logarithmic {C}ardy case: boundary states and annuli.
	\newblock {\em Nuclear Phys. B}, 930:287--327, 2018.
	
	\bibitem{frs1}
	J.~Fuchs, I.~Runkel, and C.~Schweigert.
	\newblock T{FT} construction of {RCFT} correlators. {I}. {P}artition functions.
	\newblock {\em Nuclear Phys. B}, 646(3):353--497, 2002.
	
	\bibitem{frs2}
	J.~Fuchs, I.~Runkel, and C.~Schweigert.
	\newblock T{FT} construction of {RCFT} correlators. {II}. {U}noriented world
	sheets.
	\newblock {\em Nuclear Phys. B}, 678(3):511--637, 2004.
	
	\bibitem{frs3}
	J.~Fuchs, I.~Runkel, and C.~Schweigert.
	\newblock T{FT} construction of {RCFT} correlators. {III}. {S}imple currents.
	\newblock {\em Nuclear Phys. B}, 694(3):277--353, 2004.
	
	\bibitem{frs4}
	J.~Fuchs, I.~Runkel, and C.~Schweigert.
	\newblock T{FT} construction of {RCFT} correlators. {IV}. {S}tructure constants
	and correlation functions.
	\newblock {\em Nuclear Phys. B}, 715(3):539--638, 2005.
	
	\bibitem{csrcft}
	J.~Fuchs, I.~Runkel, and C.~Schweigert.
	\newblock Twenty five years of two-dimensional rational conformal field theory.
	\newblock {\em J. Math. Phys.}, 51(1):015210, 19, 2010.
	
	\bibitem{fss}
	J.~Fuchs, G.~Schaumann, and C.~Schweigert.
	\newblock Eilenberg-{W}atts calculus for finite categories and a bimodule
	{R}adford {$S^4$} theorem.
	\newblock {\em Trans. Amer. Math. Soc.}, 373(1):1--40, 2020.
	
	\bibitem{jfcs}
	J.~Fuchs and C.~Schweigert.
	\newblock Consistent systems of correlators in non-semisimple conformal field
	theory.
	\newblock {\em Adv. Math.}, 307:598--639, 2017.
	
	\bibitem{fuchsstigner}
	J.~Fuchs and C.~Stigner.
	\newblock On {F}robenius algebras in rigid monoidal categories.
	\newblock {\em Arab. J. Sci. Eng. Sect. C Theme Issues}, 33(2):175--191, 2008.
	
	\bibitem{gk}
	E.~Getzler and M.~M. Kapranov.
	\newblock Cyclic operads and cyclic homology.
	\newblock In {\em Geometry, topology, \& physics}, volume~IV of {\em Conf.
		Proc. Lecture Notes Geom. Topology}, pages 167--201. Int. Press, Cambridge,
	MA, 1995.
	
	\bibitem{gkmod}
	E.~Getzler and M.~M. Kapranov.
	\newblock Modular operads.
	\newblock {\em Compositio Math.}, 110(1):65--126, 1998.
	
	\bibitem{godfard2025semisimplicityconformalblocks}
	P.~Godfard.
	\newblock Semisimplicity of conformal blocks, 2025.
	\newblock \href{https://arxiv.org/pdf/2507.06318}{arXiv:2507.06318}.
	
	\bibitem{skeinfin}
	S.~Gunningham, D.~Jordan, and P.~Safronov.
	\newblock The finiteness conjecture for skein modules.
	\newblock {\em Invent. Math.}, 232(1):301--363, 2023.
	
	\bibitem{henriques}
	A.~Henriques.
	\newblock Letter on “{G}eometric construction of modular functors from
	conformal field theory” and “{C}onstruction of the
	{R}eshetikhin--{T}uraev {TQFT} from conformal field theory” by {J}ørgen
	{A}ndersen and {K}enji {U}eno, 2025.
	\newblock http://andreghenriques.com/PDF/AndersenUeno.pdf.
	
	\bibitem{HL:IV}
	Y.-Z. Huang.
	\newblock A theory of tensor products for module categories for a vertex
	operator algebra. {IV}.
	\newblock {\em J. Pure Appl. Algebra}, 100(1-3):173--216, 1995.
	
	\bibitem{huang:2005:verlinde}
	Y.-Z. Huang.
	\newblock Vertex operator algebras, the {V}erlinde conjecture, and modular
	tensor categories.
	\newblock {\em Proc. Natl. Acad. Sci. USA}, 102(15):5352--5356, 2005.
	
	\bibitem{huang:rigidity}
	Y.-Z. Huang.
	\newblock Rigidity and modularity of vertex tensor categories.
	\newblock {\em Commun. Contemp. Math.}, 10:871--911, 2008.
	
	\bibitem{huangverlinde}
	Y.-Z. Huang.
	\newblock {Vertex operator algebras and the {V}erlinde Conjecture}.
	\newblock {\em Comm. Contemp. Math.}, 10(1):103--154, 2008.
	
	\bibitem{HL:I}
	Y.-Z. Huang and J.~Lepowsky.
	\newblock A theory of tensor products for module categories for a vertex
	operator algebra. {I}.
	\newblock Number 883, pages 148--203. 1994.
	\newblock Geometric aspects of infinite integrable systems (Japanese) (Kyoto,
	1993).
	
	\bibitem{HL:II}
	Y.-Z. Huang and J.~Lepowsky.
	\newblock A theory of tensor products for module categories for a vertex
	operator algebra. {I}, {II}.
	\newblock {\em Selecta Math. (N.S.)}, 1(4):699--756, 757--786, 1995.
	
	\bibitem{HL:III}
	Y.-Z. Huang and J.~Lepowsky.
	\newblock A theory of tensor products for module categories for a vertex
	operator algebra. {III}.
	\newblock {\em J. Pure Appl. Algebra}, 100(1-3):141--171, 1995.
	
	\bibitem{HLZ}
	Y.-Z. Huang, J.~Lepowsky, and L.~Zhang.
	\newblock Logarithmic tensor category theory, {I-VIII}.
	\newblock \href{https://arxiv.org/pdf/1012.4193}{arXiv:1012.4193},
	\href{https://arxiv.org/pdf/arXiv:1012.4196}{arXiv:1012.4196},
	\href{https://arxiv.org/pdf/arXiv:1012.4197}{arXiv:1012.4197},
	\href{https://arxiv.org/pdf/arXiv:1012.4198}{arXiv:1012.4198},
	\href{https://arxiv.org/pdf/arXiv:1012.4199}{arXiv:1012.4199},
	\href{https://arxiv.org/pdf/arXiv:1012.4202}{arXiv:1012.4202},
	\href{https://arxiv.org/pdf/arXiv:1110.1929}{arXiv:1110.1929},
	\href{https://arxiv.org/pdf/arXiv:1110.1931}{arXiv:1110.1931}.
	
	\bibitem{kapustinsaulina}
	A.~Kapustin and N.~Saulina.
	\newblock Surface operators in 3d topological field theory and 2d rational
	conformal field theory.
	\newblock In {\em Mathematical foundations of quantum field theory and
		perturbative string theory}, volume~83 of {\em Proc. Sympos. Pure Math.},
	pages 175--198. Amer. Math. Soc., Providence, RI, 2011.
	
	\bibitem{Lepowsky.Li}
	J.~Lepowsky and H.~Li.
	\newblock {\em Introduction to vertex operator algebras and their
		representations}, volume 227 of {\em Progress in Mathematics}.
	\newblock Birkh\"auser Boston, Inc., Boston, MA, 2004.
	
	\bibitem{higheralgebra}
	J.~Lurie.
	\newblock Higher algebra.
	\newblock Available at https://www.math.ias.edu/\~{}lurie/papers/HA.pdf.
	
	\bibitem{lyu}
	V.~Lyubashenko.
	\newblock Modular transformations for tensor categories.
	\newblock {\em J. Pure Appl. Algebra}, 98(3):279--327, 1995.
	
	\bibitem{lyulex}
	V.~Lyubashenko.
	\newblock Ribbon abelian categories as modular categories.
	\newblock {\em J. Knot Theory Ramifications}, 5(3):311--403, 1996.
	
	\bibitem{lyubacmp}
	V.~V. Lyubashenko.
	\newblock Invariants of {$3$}-manifolds and projective representations of
	mapping class groups via quantum groups at roots of unity.
	\newblock {\em Comm. Math. Phys.}, 172(3):467--516, 1995.
	
	\bibitem{mwskein}
	L.~Müller and L.~Woike.
	\newblock Admissible skein modules and ansular functors: A comparison.
	\newblock arXiv:2409.17047 [math.QA], 2024.
	
	\bibitem{ms89}
	G.~Moore and N.~Seiberg.
	\newblock Classical and quantum conformal field theory.
	\newblock {\em Comm. Math. Phys.}, 123(2):177--254, 1989.
	
	\bibitem{mooreseiberg}
	G.~Moore and N.~Seiberg.
	\newblock Lectures on {RCFT}.
	\newblock In {\em Physics, geometry, and topology ({B}anff, {AB}, 1989)},
	volume 238 of {\em NATO Adv. Sci. Inst. Ser. B: Phys.}, pages 263--361.
	Plenum, New York, 1990.
	
	\bibitem{moriwaki:2024:vertexoperatoralgebraparenthesized}
	Y.~Moriwaki.
	\newblock Vertex operator algebra and parenthesized braid operad, 2024.
	\newblock \href{https://arxiv.org/pdf/2209.10443}{arXiv:2209.10443}.
	
	\bibitem{cyclic}
	L.~M\"uller and L.~Woike.
	\newblock Cyclic framed little disks algebras, {G}rothendieck-{V}erdier duality
	and handlebody group representations.
	\newblock {\em Q. J. Math.}, 74(1):163--245, 2023.
	
	\bibitem{10.1093/imrn/rnad178}
	L.~M\"uller and L.~Woike.
	\newblock Classification of consistent systems of handlebody group
	representations.
	\newblock {\em Int. Math. Res. Not. IMRN}, (6):4767--4803, 2024.
	
	\bibitem{mwcenter}
	L.~M\"uller and L.~Woike.
	\newblock The distinguished invertible object as ribbon dualizing object in the
	{D}rinfeld center.
	\newblock {\em Selecta Math. (N.S.)}, 30(5):Paper No. 98, 27, 2024.
	
	\bibitem{envas}
	L.~M\"uller and L.~Woike.
	\newblock Categorified open topological field theories.
	\newblock {\em Proc. Amer. Math. Soc.}, 153(6):2381--2396, 2025.
	
	\bibitem{NT}
	K.~Nagatomo and A.~Tsuchiya.
	\newblock Conformal field theories associated to regular chiral vertex operator
	algebras. {I}. {T}heories over the projective line.
	\newblock {\em Duke Math. J.}, 128(3):393--471, 2005.
	
	\bibitem{rt2}
	N.~Reshetikhin and V.~G. Turaev.
	\newblock Invariants of {$3$}-manifolds via link polynomials and quantum
	groups.
	\newblock {\em Invent. Math.}, 103(3):547--597, 1991.
	
	\bibitem{rt1}
	N.~Y. Reshetikhin and V.~G. Turaev.
	\newblock Ribbon graphs and their invariants derived from quantum groups.
	\newblock {\em Comm. Math. Phys.}, 127(1):1--26, 1990.
	
	\bibitem{salvatorewahl}
	P.~Salvatore and N.~Wahl.
	\newblock Framed discs operads and {B}atalin-{V}ilkovisky algebras.
	\newblock {\em Q. J. Math.}, 54(2):213--231, 2003.
	
	\bibitem{schommerpries}
	C.~J. Schommer-Pries.
	\newblock {\em The classification of two-dimensional extended topological field
		theories}.
	\newblock PhD thesis, Berkeley, 2009.
	
	\bibitem{Segal}
	G.~Segal.
	\newblock Two-dimensional conformal field theories and modular functors.
	\newblock In {\em I{X}th {I}nternational {C}ongress on {M}athematical {P}hysics
		({S}wansea, 1988)}, pages 22--37. Hilger, Bristol, 1989.
	
	\bibitem{tillmann}
	U.~Tillmann.
	\newblock {$\mathscr{S}$}-structures for {$k$}-linear categories and the
	definition of a modular functor.
	\newblock {\em J. London Math. Soc. (2)}, 58(1):208--228, 1998.
	
	\bibitem{Tsuchimoto}
	Y.~Tsuchimoto.
	\newblock On the coordinate-free description of the conformal blocks.
	\newblock {\em J. Math. Kyoto Univ.}, 33(1):29--49, 1993.
	
	\bibitem{turaev}
	V.~G. Turaev.
	\newblock {\em Quantum invariants of knots and 3-manifolds}, volume~18 of {\em
		De Gruyter Studies in Mathematics}.
	\newblock Walter de Gruyter \& Co., Berlin, revised edition, 2010.
	
	\bibitem{verlinde}
	E.~Verlinde.
	\newblock Fusion rules and modular transformations in {$2$}d conformal field
	theory.
	\newblock {\em Nuclear Phys. B}, 300(3):360--376, 1988.
	
	\bibitem{Walker}
	K.~Walker.
	\newblock {TQFT}s.
	\newblock Notes available at \url{http://canyon23.net/math/tc.pdf}.
	
	\bibitem{Witten:1988hf}
	E.~Witten.
	\newblock Quantum field theory and the {J}ones polynomial.
	\newblock {\em Comm. Math. Phys.}, 121(3):351--399, 1989.
	
	\bibitem{woike2024cyclicmodularmicrocosmprinciple}
	L.~Woike.
	\newblock The cyclic and modular microcosm principle in quantum topology, 2024.
	\newblock \href{https://arxiv.org/pdf/2408.02644}{arXiv:2408.02644} Accepted
	for publication in \emph{Canad. J. Math.}
	
	\bibitem{zhu:global}
	Y.~Zhu.
	\newblock Global vertex operators on {R}iemann surfaces.
	\newblock {\em Comm. Math. Phys.}, 165(3):485--531, 1994.
	
\end{thebibliography}
\end{document}